\documentclass{article}
\usepackage{amsmath,amssymb,amsthm}
\usepackage{graphicx,url,color}
\usepackage{mathrsfs}
\usepackage[colorlinks=true]{hyperref}
\usepackage{pdfsync}

\def\R{\mathrm{I\kern-0.21emR}}
\def\N{\mathrm{I\kern-0.21emN}}

\newcommand{\Ran}{\operatorname{Ran}}
\newcommand{\Span}{\operatorname{span}}
\newcommand{\dist}{\operatorname{dist}}
\newcommand{\supp}{\operatorname{supp}}
\newcommand{\Id}{\mathrm{Id}}

\renewcommand{\geq}{\geqslant}
\renewcommand{\leq}{\leqslant}

\newtheorem{theorem}{Theorem}[section]
\newtheorem{proposition}[theorem]{Proposition}
\newtheorem{lemma}[theorem]{Lemma}
\newtheorem{corollary}[theorem]{Corollary}
\theoremstyle{definition}

\newtheorem{example}[theorem]{Example}
\newtheorem{remark}[theorem]{Remark}

\title{Null-controllable data for cascade heat equations:\\ almost sharp spectral spaces and modal critical times}

\author{
Hugo Lhachemi\footnote{\raggedright Universit\'e Paris-Saclay, CNRS, CentraleSup\'elec, Laboratoire des signaux et syst\`emes, 91190 Gif-sur-Yvette, France (\texttt{hugo.lhachemi@centralesupelec.fr}).\endgraf}
\and
Christophe Prieur\footnote{\raggedright Universit\'e Grenoble Alpes, CNRS, Grenoble INP, GIPSA-lab, F-38000 Grenoble, France (\texttt{christophe.prieur@gipsa-lab.fr}).\endgraf}
\and
Emmanuel Tr\'elat\footnote{\raggedright Sorbonne Universit\'e, Universit\'e Paris Cit\'e, CNRS, Inria, Laboratoire Jacques-Louis Lions, LJLL, F-75005 Paris, France (\texttt{emmanuel.trelat@sorbonne-universite.fr}).\endgraf}
}

\date{}

\begin{document}

\maketitle

\begin{abstract}
We study one-dimensional cascades of heat equations controlled through their last component. For a fixed time $T>0$, we consider the Hilbert space $\mathcal{C}_T$ of all initial data that can be steered exactly to zero at time $T$, endowed with the norm of the minimal $L^2$ control. This space can be strongly anisotropic: in explicit examples, it is incomparable with every isotropic spectral Sobolev or sub-Gaussian Gevrey scale associated with an adapted basis.

For two equations, under the block Fattorini--Hautus condition, we give an almost sharp spectral description of $\mathcal{C}_T$. More precisely, in an adapted Riesz basis of generalized eigenvectors we define an explicit increasing scale $(\mathcal{V}_\tau)_{\tau\in\R}$ of Hilbert spaces such that $\mathcal{V}_{T-\varepsilon}\hookrightarrow\mathcal{C}_T\hookrightarrow\mathcal{V}_T$ for every $\varepsilon>0$.
The weights of $\mathcal{V}_T$ combine heat dissipation with the loss of observation inside each double generalized eigenspace. When the coupling and control regions are disjoint, this recovers the known minimal null-control time $T^\star$ and identifies it as a transition time for controllable data: for $T<T^\star$, $\mathcal{C}_T$ and the natural state space are incomparable, whereas for $T>T^\star$, the natural state space is properly contained in $\mathcal{C}_T$.

We extend the analysis to arbitrary finite upper triangular heat cascades. The observation losses of the different directions in each spectral block may have distinct exponential orders, leading to a hierarchy of critical times rather than a single threshold. This hierarchy describes successive gains of null controllability on adapted modal subspaces and accounts for the interactions and cancellations between different coupling paths.
\end{abstract}

\medskip


\smallskip


\section{Introduction}
Let $T>0$, let $\omega = (a,b)$ be a nonempty open interval compactly contained in $(0,\pi)$, and let $\beta\in L^\infty(0,\pi)$ be real-valued function. We consider
\begin{equation}\label{system}
\left\{
\begin{array}{ll}
\partial_t y_1 = \partial_{xx}y_1 + \beta\, y_2 & \text{in }(0,T)\times(0,\pi),\\
\partial_t y_2 = \partial_{xx}y_2 + 1_\omega u & \text{in }(0,T)\times(0,\pi),\\
y_1(t,0) = y_1(t,\pi) = y_2(t,0) = y_2(t,\pi) = 0 & \text{for }t\in(0,T),\\
y(0,\cdot) = y^0. &
\end{array}
\right.
\end{equation}
The natural energy space is $\mathcal{H} = L^2(0,\pi;\R^2)$, and the control $u$ belongs to $L^2(0,T;L^2(\omega))$. The first component is reached only through the coupling $\beta y_2$.

Under standard nondegeneracy assumptions, including sign conditions on a coupling region reached by the control, Carleman estimates for parabolic cascades give null controllability in every positive time \cite{AKBD_JEE2009,AKBGBdT_MCRF2011,GBdT_PM2010,GBPG_AA2006}. A different phenomenon appears when
\begin{equation}\label{disjoint}
\supp(\beta)\cap\omega = \varnothing.
\end{equation}
For this geometry, \cite{AKBGBdT_CRM2014,AKBGBdT_JMAA2016} show that null controllability in $\mathcal{H}$ may have a minimal time $T^\star\in[0, +\infty]$: it holds for $T>T^\star$ and fails for $T<T^\star$.

Related positive minimal-time phenomena occur for pointwise control \cite{Dolecki_SM1973}, through condensation of spectral branches \cite{AKBGBdT_JFA2014}, and for Grushin-type equations \cite{BeauchardDardeErvedoza_AIF2020,BMM_JDE2015,DardeKoenigRoyer_AHL2023}. For a boundary-controlled system of two one-dimensional parabolic equations with unequal diffusivities, Diophantine properties of their ratio may decide null controllability \cite{LucaDeTeresa_SeMA2013}. The quantitative Fattorini--Hautus theory of \cite{AKBGBM_JMPA2019}, related to the resolvent viewpoint of \cite{DuyckaertsMiller_JFA2012}, places these phenomena in a common modal framework; \cite{Ouaili_MCRF2020,Samb_COCV2021} treat simultaneous degeneracies of eigenvalues and eigenfunctions.

Null controllability in sufficiently large time for the one-dimensional structurally damped wave equation with a moving control was proved in \cite{MartinRosierRouchon_SICON2013}. Sharp thresholds were subsequently identified for parabolic--transport systems \cite{BeauchardKoenigLeBalch_JEP2020} and for the hyperbolic component of wave--heat cascades \cite{LPT_arXiv2026}. Underactuated constant-coefficient parabolic--transport systems are treated in \cite{KoenigLissy_AIF2026}, with a spectral Kalman characterization for sufficiently regular data in large time and Sobolev obstructions in small time. For the parabolic--transport and wave--heat models, the thresholds are determined, respectively, by the transport component and by finite propagation of the wave. In the present equal-diffusion system, by contrast, the time exponentials themselves create no threshold: the obstruction comes from directions that become almost invisible inside the double eigenspaces.

The number $T^\star$ answers an all-or-nothing question about $\mathcal{H}$. It does not describe which data remain controllable below $T^\star$, or which more singular data become controllable above it. We therefore fix $T$ and study the Hilbert space of data that are exactly null controllable at time $T$,
$$
\mathcal{C}_T = \{y^0\ \mid\ y^0\text{ can be steered to zero at time }T\},
$$
defined below as a subspace of a natural extrapolation Hilbert space and equipped with the norm of the minimal control. The family $(\mathcal{C}_T)_{T>0}$ is increasing, with $\Vert y^0\Vert_{\mathcal{C}_{T_2}}\leq\Vert y^0\Vert_{\mathcal{C}_{T_1}}$ whenever $T_1<T_2$: a control reaching zero at time $T_1$, extended by zero on $(T_1,T_2)$, reaches zero at time $T_2$. For a Hilbert space $X$ of initial data, we define the critical time of $X$ by
\begin{equation}\label{def_TX}
T_X = \inf\{T>0\,\mid\, X\hookrightarrow\mathcal{C}_T\},
\end{equation}
with the convention $\inf\varnothing=+\infty$. This critical time is the infimum of the horizons for which the continuous inclusion $X\hookrightarrow\mathcal{C}_T$ holds. For $X=\mathcal{H}$, this is the classical minimal null-control time of \cite{AKBGBdT_CRM2014,AKBGBdT_JMAA2016}: the minimal time depends on the space of initial data under consideration, and \eqref{def_TX} makes this dependence explicit. Below, $X$ will be the energy space, a heat scale, one of the spaces $\mathcal{V}_\tau$, a modal subspace, or the line spanned by a single datum. We say ``critical time'' rather than ``minimal time'' to stress that $T_X$ is a threshold in the increasing family $(\mathcal{C}_T)_{T>0}$, and to reserve the latter name for $T^\star$.

\medskip

The spectral description starts from the Dirichlet eigenfunctions and the coupling moment
\begin{equation}\label{def_Ik}
\varphi_k(x) = \sqrt{2/\pi}\sin(kx), \qquad I_k = \int_0^\pi\beta(x)\varphi_k(x)^2\, dx.
\end{equation}
The eigenvalue $-k^2$ of the coupled generator has algebraic multiplicity two. After choosing an adapted Jordan basis, the restriction of the adjoint observation to this generalized eigenspace is quantified by
$$
\rho_k = \big(I_k^2\Vert\varphi_k\Vert_{L^2(\omega)}^2 + \Vert\theta_k\Vert_{L^2(\omega)}^2\big)^{1/2},
$$
where $\theta_k$ is constructed in Subsection \ref{spectralblocks} and is orthogonal to $\varphi_k$ on $\omega$. The smaller $\rho_k$ is, the more costly it is to distinguish the two generalized eigenvectors by observing the second component on $\omega$. Throughout the two-equation results we assume the block Fattorini--Hautus condition $\rho_k>0$ for every $k$, equivalently approximate controllability; see \eqref{FH}. We then set
\begin{equation}\label{Tstarintro}
T^\star = \limsup_{k\to +\infty}\frac{-\log\rho_k}{k^2}\in[0,+\infty].
\end{equation}
Proposition \ref{Hcomparison} proves that $T^\star = T_{\mathcal{H}}$ in the notation \eqref{def_TX}: for every coupling $\beta$ and every control interval $\omega$ satisfying \eqref{FH}, the number $T^\star$ is the critical time of the energy space. Under \eqref{disjoint}, $\rho_k$ is equivalent, up to one power of $k$, to the geometric moments of \cite{AKBGBdT_JMAA2016,BoyerMorancey_AMBP2025}, so that $T^\star$ coincides with the minimal time of \cite{AKBGBdT_JMAA2016}; this justifies the notation. The infimum in \eqref{def_TX} does not decide the inclusion when $T = T^\star$.

\medskip
Our first main result introduces the explicit Hilbert scale
$$
\Vert y\Vert_{\mathcal{V}_\tau}^2
= \sum_{k\geq1}e^{-2\tau k^2}
\left(\frac{\vert c_{1,k}\vert^2}{\rho_k^2}+\vert c_{2,k}\vert^2\right),
\qquad
y=\sum_{k\geq1}(c_{1,k}\Phi_{1,k}+c_{2,k}\Phi_{2,k}).
$$
Here $(\Phi_{1,k},\Phi_{2,k})_{k\geq1}$ is the adapted Riesz basis of the coupled generator constructed in Subsection \ref{spectralblocks}. We prove that
\begin{equation}\label{sandwichintro}
\mathcal{V}_{T-\varepsilon}\hookrightarrow\mathcal{C}_T\hookrightarrow\mathcal{V}_T \qquad \forall\varepsilon>0 . 
\end{equation}
The two inclusions give an explicit, almost sharp spectral description of the Hilbert space $\mathcal{C}_T$ of null-controllable data: the necessary and sufficient weights have the same leading Gaussian factor $e^{-2Tk^2}$ and the same anisotropic factor $\rho_k^{-2}$. They differ only by factors of order $e^{\mathrm{o}(k^2)}$, which can be absorbed by replacing $T$ with $T-\varepsilon$ in the sufficient condition. A one-frequency test proves, in particular, that neither $e^{-2Tk^2}$ nor $\rho_k^{-2}$ can be improved at the leading exponential scale. The common analytic tool is Proposition \ref{twosidedmoment}: for every fixed multiplicity $N$, it combines a vector-valued M\"untz--Sz\'asz lower estimate for $(t^me^{-k^2t})_{k\geq1,\,0\leq m<N}$ with a Hilbert--Schur upper estimate and a sharper upper bound at a single frequency \cite{AKBGBdT_JMPA2011,Boyer_Notes2023,BoyerMorancey_CRM2023,FCGBdT_JFA2010}. The first theorem uses its specialization to $N=2$.

The comparison \eqref{sandwichintro} turns $T^\star$ into a transition time for the nature of the controllable-data space. If $T<T^\star$, then $\mathcal{C}_T$ contains data more singular than $L^2$ in the directly controlled component, but excludes some $L^2$ data in the weakly observed direction; thus $\mathcal{C}_T$ and $\mathcal{H}$ are incomparable. If $T>T^\star$, then $\mathcal{H}$ is properly contained in $\mathcal{C}_T$. The arithmetic example in Subsection \ref{sec_3examples} shows that below $T^\star$ no isotropic spectral Sobolev or sub-Gaussian Gevrey norm in the adapted basis can describe $\mathcal{C}_T$.

\medskip
The spectral weights express a simple competition. During the free evolution, the $k$-th frequency is damped by $e^{-Tk^2}$; observing its weak generalized-eigenvector direction costs a factor of order $\rho_k^{-1}$. A vector moment estimate shows that, up to subquadratic losses, these two effects combine multiplicatively. This is the origin of the weight $e^{-2Tk^2}\rho_k^{-2}$ in the squared norm of $\mathcal{V}_T$.

\medskip
Our second main result concerns $N$ heat equations with a bounded upper triangular coupling matrix. On the $k$-th generalized eigenspace, the nilpotent part $J_k$ satisfies
$e^{tJ_k^*}=\sum_{m=0}^{N-1}\frac{t^m}{m!}(J_k^*)^m$.
The observed adjoint solution is therefore $e^{-k^2t}$ times a polynomial of degree at most $N-1$. Its coefficient map has $N$ singular values
$\mu_{1,k}\geq\cdots\geq\mu_{N,k}>0$. The $i$-th singular directions, taken over all frequencies, define a modal subspace $X_i$ of the initial-data space. Their observation loss on the Gaussian scale $e^{-Tk^2}$ defines a number $\tau_i$, which Theorem \ref{modalcriticalN} identifies with the critical time of $X_i$. For $N=2$, the two times are $0$ and $T^\star$. For $N\geq3$, the intermediate values $\tau_2,\ldots,\tau_{N-1}$ distinguish modal subspaces that the single threshold $\tau_N$ of the full energy space cannot distinguish. Already for three equations, several directed coupling paths and cancellations between observed coefficients may alter these values, so consecutive coupling moments alone are not sufficient in general.

\medskip

Weighted modal spaces of controllable data for parabolic equations with a fixed spatial control function or point control occur in \cite{MicuZuazua_SICON2006,Zuazua_JEMS2011}; regularity of HUM minimizers and minimal boundary controls is studied in \cite{MicuZuazua_SCL2011}; the Grushin setting is treated in \cite{BMM_JDE2015}. For the critical half-heat equation, almost-sharp conditions on null-controllable initial states, together with the time-independence of their space, are obtained in \cite{HartmannKoenig_arXiv2025}. The block moment method of \cite{BBM_AHL2020,BoyerMorancey_CRM2023} defines minimal times for a datum and for a closed subspace through finite-dimensional block costs. For the present two-equation model, \cite{BoyerMorancey_AMBP2025} computes the double-block denominator used below and reduces it to geometric moments when the coupling and control regions are disjoint. We use this modal information to obtain two-sided Hilbert-space estimates for the canonical range at every horizon and to compare that range with $L^2$.

Terminal reachable states for the scalar heat equation form a related but different problem. Boundary reachable sets and their holomorphic structure are studied in \cite{DardeErvedoza_SICON2018,HartmannKellayTucsnak_JEMS2020,HartmannOrsoni_JMPA2021,KellayNormandTucsnak_APDE2022,MartinRosierRouchon_AMRX2016,Orsoni_JFA2021,StrohmaierWaters_MZ2022}. Distributed controls and perturbations are treated in \cite{ChenRosier_CRM2022,ErvedozaLeBalchTucsnak_JFA2022}. Here the terminal target is zero and the object to be described is the set of admissible initial data. The range-theoretic work \cite{Lissy_JMAA2026} is also relevant: it treats hybrid range inclusions for a constant-coupling equal-diffusion cascade with a mixed null and approximate objective.

\paragraph{Structure of the paper.}
Sections \ref{twocascade} and \ref{finitecascades} treat the two-equation case and finite cascades, respectively. Both follow the same progression: spectral preliminaries, the fixed-horizon theorem, critical times and comparison with heat scales, proof of the theorem, then examples. The operator-range framework is introduced once, in Subsection \ref{sec_range_duality}, and reused for finite cascades. Subsections \ref{spectralphenomena} and \ref{sec_3examples} treat the disjoint geometry and nonstandard examples for two equations; the general case concludes with remarks on the three-equation block and a nonlocal inverse-design construction. Appendix \ref{appendix} contains the common moment estimate (Subsection \ref{appendixmoments}) and the proof of the adapted block-basis lemma (Subsection \ref{app_proof_lem_rieszN}).

\section{The two-equation cascade}\label{twocascade}
\subsection{Spectral preliminaries}\label{spectralblocks}
\paragraph{Generator and adapted Jordan bases.}
System \eqref{system} is written in the abstract form $y' = Ay + Bu$ with
$$
A = \begin{pmatrix}\partial_{xx} & \beta \\ 0 & \partial_{xx} \end{pmatrix}, \qquad Bu = \begin{pmatrix}0\\1_\omega u\end{pmatrix},
$$
with $D(A) = H^2(0,\pi;\R^2)\cap H_0^1(0,\pi;\R^2)$. The operator $A$ generates an analytic semigroup on $\mathcal{H} = L^2(0,\pi;\R^2)$, and $B\in L(L^2(0,\pi),\mathcal{H})$. The adjoint observation operator is $B^*(f,g) = g_{\vert\omega}$.

For $T>0$, let
\begin{equation}\label{globalgramian}
L_Tu=\int_0^T e^{(T-t)A}Bu(t)\,dt,
\qquad
G_T=L_TL_T^*=\int_0^T e^{(T-t)A}BB^*e^{(T-t)A^*}\,dt.
\end{equation}
The controllability Gramian $G_T$ is a bounded, nonnegative selfadjoint operator on $\mathcal{H}$, and
\begin{equation}\label{gramianobservationidentity}
\langle G_Tz,z\rangle_{\mathcal{H}}
=\Vert L_T^*z\Vert_{L^2(0,T;L^2(\omega))}^2
=\int_0^T \Vert B^* e^{sA^*}z\Vert_{L^2(\omega)}^2\, ds.
\end{equation}

For every $k \geq 1$, recalling that $\varphi_k$ and $I_k$ are defined by \eqref{def_Ik}, let $\psi_k$ be the unique solution of
\begin{equation}\label{psik}
-\psi_k''-k^2\psi_k = (\beta-I_k)\varphi_k,\qquad \psi_k(0) = \psi_k(\pi) = 0,\qquad \langle\psi_k,\varphi_k\rangle_{L^2(0,\pi)} = 0.
\end{equation}
Indeed, writing $(\beta-I_k)\varphi_k=\sum_{n\geq1}f_n\varphi_n$, the definition \eqref{def_Ik} of $I_k$ gives $f_k=\langle\beta\varphi_k,\varphi_k\rangle_{L^2(0,\pi)}-I_k=0$, and the equation in \eqref{psik} amounts to $(n^2-k^2)\langle\psi_k,\varphi_n\rangle_{L^2(0,\pi)}=f_n$ for every $n\geq1$. The $k$-th of these identities is satisfied whatever the value of $\langle\psi_k,\varphi_k\rangle_{L^2(0,\pi)}$, which the last condition in \eqref{psik} sets to zero, while the other coefficients are uniquely determined. Since $\vert n^2-k^2\vert=\vert n-k\vert(n+k)\geq2k-1$ for $n\neq k$, the resulting series converges in $L^2(0,\pi)$ and
$$
\Vert\psi_k\Vert_{L^2(0,\pi)}\leq\frac{\Vert(\beta-I_k)\varphi_k\Vert_{L^2(0,\pi)}}{2k-1}\leq\frac{2\Vert\beta\Vert_{L^\infty(0,\pi)}}{2k-1} ,
$$
using $\Vert\varphi_k\Vert_{L^2(0,\pi)}=1$ and $\vert I_k\vert\leq\Vert\beta\Vert_{L^\infty(0,\pi)}$. Moreover $\psi_k\in H^2(0,\pi)\cap H_0^1(0,\pi)$, since $n^2/(n^2-k^2)$ is bounded with respect to $n\neq k$. Setting
\begin{equation}\label{deltatheta}
\delta_k = \frac{\langle\psi_k,\varphi_k\rangle_{L^2(\omega)}}{\Vert\varphi_k\Vert_{L^2(\omega)}^2},\qquad \theta_k = \psi_k-\delta_k\varphi_k,
\end{equation}
one has $\langle\theta_k,\varphi_k\rangle_{L^2(\omega)} = 0$. Since $\Vert\varphi_k\Vert_{L^2(\omega)}^2$ is positive for every $k$ and converges to $\frac{b-a}{\pi}$ as $k\to+\infty$, there exists $m_\omega>0$ such that
\begin{equation}\label{alphabounds}
m_\omega \leq \alpha_k = \Vert\varphi_k\Vert_{L^2(\omega)}^2 \leq 1.
\end{equation}
The Cauchy--Schwarz inequality and \eqref{alphabounds} give $\vert\delta_k\vert\leq\Vert\psi_k\Vert_{L^2(\omega)}/\sqrt{\alpha_k}\leq\Vert\psi_k\Vert_{L^2(0,\pi)}/\sqrt{m_\omega}$, and then $\Vert\theta_k\Vert_{L^2(0,\pi)}\leq\Vert\psi_k\Vert_{L^2(0,\pi)}+\vert\delta_k\vert$. With the bound on $\Vert\psi_k\Vert_{L^2(0,\pi)}$ obtained above, this gives
\begin{equation}\label{psiestimates}
\Vert\psi_k\Vert_{L^2(0,\pi)} + \vert\delta_k\vert + \Vert\theta_k\Vert_{L^2(0,\pi)} \leq \frac{C}{k}.
\end{equation}

\begin{lemma}[Adapted Jordan bases]\label{lem_Jordan}
For $k \geq 1$, set
$$
\Phi_{1,k} = \begin{pmatrix}\varphi_k\\0\end{pmatrix},\qquad \Phi_{2,k} = \begin{pmatrix}\psi_k + \delta_k\varphi_k\\\varphi_k\end{pmatrix},\qquad \Psi_{1,k} = \begin{pmatrix}\varphi_k\\\theta_k\end{pmatrix},\qquad \Psi_{2,k} = \begin{pmatrix}0\\\varphi_k\end{pmatrix}.
$$
The families $(\Phi_{1,k},\Phi_{2,k})_{k \geq 1}$ and $(\Psi_{1,k},\Psi_{2,k})_{k \geq 1}$ are dual Riesz bases of $\mathcal{H}$. We also write $\Phi_kc=c_1\Phi_{1,k}+c_2\Phi_{2,k}$ and $\Psi_ka=a_1\Psi_{1,k}+a_2\Psi_{2,k}$ for the corresponding block maps. These vectors satisfy
\begin{equation}\label{Jordanrelations}
\begin{aligned}
(A + k^2)\Phi_{1,k}& = 0, & (A + k^2)\Phi_{2,k}& = I_k\Phi_{1,k},\\
(A^*+k^2)\Psi_{1,k}& = I_k\Psi_{2,k}, \quad& (A^*+k^2)\Psi_{2,k}& = 0.
\end{aligned}
\end{equation}
Moreover,
\begin{equation}\label{observedbasis}
B^*\Psi_{1,k} = \theta_{k\vert\omega},\qquad B^*\Psi_{2,k} = \varphi_{k\vert\omega},\qquad \langle\theta_k,\varphi_k\rangle_{L^2(\omega)} = 0.
\end{equation}
\end{lemma}

\begin{proof}
The relations in \eqref{Jordanrelations} follow from \eqref{psik}; adding a multiple of $\varphi_k$ to the first component of $\Phi_{2,k}$ does not change them. By the definition of $\delta_k$, the two observed functions $B^*\Psi_{1,k}=\theta_{k\vert\omega}$ and $B^*\Psi_{2,k}=\varphi_{k\vert\omega}$ are orthogonal in $L^2(\omega)$, which is \eqref{observedbasis}.

We apply the triangular criterion of Lemma \ref{lem_triangular}, proved in Subsection \ref{app_triangular} of the appendix, with $\mathcal{X}_1=L^2(0,\pi)\times\{0\}$ and $\mathcal{X}_2=\{0\}\times L^2(0,\pi)$, with the orthonormal basis $e_{1,k}=(\varphi_k,0)$, $e_{2,k}=(0,\varphi_k)$ of $\mathcal{H}$, and with the corrections $\eta_{1,k}=0$ and $\eta_{2,k}=(\psi_k+\delta_k\varphi_k,0)$, which are square summable by \eqref{psiestimates}. Hence $(\Phi_{1,k},\Phi_{2,k})_{k\geq1}$ is a Riesz basis of $\mathcal{H}$, whose dual basis is given by $e_{j,k}-K^*e_{j,k}$, where $K(0,\varphi_k)=(\psi_k+\delta_k\varphi_k,0)$ and $K(\varphi_k,0)=0$. Since $K^*$ vanishes on $\mathcal{X}_2$, the dual vector of $\Phi_{2,k}$ is $(0,\varphi_k)=\Psi_{2,k}$. That of $\Phi_{1,k}$ is $(\varphi_k,0)-K^*(\varphi_k,0)$, where
$$
K^*(\varphi_k,0)=\sum_{j\geq1}\langle\varphi_k,\psi_j+\delta_j\varphi_j\rangle_{L^2(0,\pi)}\,(0,\varphi_j).
$$
Expanding \eqref{psik} in the Dirichlet basis gives, for $j\neq k$,
$$
\langle\varphi_k,\psi_j\rangle_{L^2(0,\pi)}=\frac{\langle\beta\varphi_j,\varphi_k\rangle_{L^2(0,\pi)}}{k^2-j^2}=-\langle\psi_k,\varphi_j\rangle_{L^2(0,\pi)},
$$
since $\beta$ is real-valued, while $\langle\varphi_k,\psi_k+\delta_k\varphi_k\rangle_{L^2(0,\pi)}=\delta_k$. As $\psi_k=\sum_{j\neq k}\langle\psi_k,\varphi_j\rangle_{L^2(0,\pi)}\varphi_j$, we obtain $K^*(\varphi_k,0)=(0,-\psi_k+\delta_k\varphi_k)$, and the dual vector of $\Phi_{1,k}$ is $(\varphi_k,\psi_k-\delta_k\varphi_k)=\Psi_{1,k}$.

Generalized eigenspaces belonging to distinct eigenvalues of a nonnormal operator need not be mutually orthogonal; the property used below is the biorthogonality of the primal and adjoint bases.
\end{proof}

For later use, the semigroup formulas are
\begin{align}
e^{tA}\Phi_{1,k} &= e^{-k^2t}\Phi_{1,k}, & e^{tA}\Phi_{2,k} &= e^{-k^2t}(\Phi_{2,k} + tI_k\Phi_{1,k}), \label{semigroupblocks} \\
e^{tA^*}\Psi_{1,k} &= e^{-k^2t}(\Psi_{1,k} + tI_k\Psi_{2,k}), & e^{tA^*}\Psi_{2,k} &= e^{-k^2t}\Psi_{2,k}. \label{adjointblocks}
\end{align}

\paragraph{The modal observation defect.}

Formula \eqref{adjointblocks} gives
$$
B^*e^{sA^*}\Psi_{1,k} = e^{-k^2s}(\theta_k+sI_k\varphi_k)_{\vert\omega},
\qquad B^*e^{sA^*}\Psi_{2,k} = e^{-k^2s}\varphi_{k\vert\omega}.
$$
Since $\theta_k$ and $\varphi_k$ are orthogonal in $L^2(\omega)$, we have
$$
\Vert B^*e^{sA^*}\Psi_{1,k}\Vert_{L^2(\omega)}^2 = e^{-2k^2s}\big(\Vert\theta_k\Vert_{L^2(\omega)}^2+s^2I_k^2\alpha_k\big),
\qquad
\Vert B^*e^{sA^*}\Psi_{2,k}\Vert_{L^2(\omega)}^2 = e^{-2k^2s}\alpha_k.
$$
The two coefficients of the observed polynomial $\theta_{k\vert\omega}+sI_k\varphi_{k\vert\omega}$ have squared norm
\begin{equation}\label{rhok}
\rho_k^2=\Vert\theta_k\Vert_{L^2(\omega)}^2+I_k^2\alpha_k=\Vert\theta_{k\vert\omega}\Vert_{L^2(\omega)}^2+\Vert I_k\varphi_{k\vert\omega}\Vert_{L^2(\omega)}^2.
\end{equation}
This coefficient norm, rather than the observation at one particular time, is the quantity appearing in the block moment estimate in Subsection~\ref{proofmain}. It measures how well the first direction of the double generalized eigenspace is seen through the constant and linear coefficients together.

Unless explicitly stated otherwise, we assume in the two-equation analysis that
\begin{equation}\label{FH}
\rho_k>0\qquad\forall k \geq 1.
\end{equation}
This is the Fattorini--Hautus condition $\ker(A^*+k^2)\cap\ker B^* = \{0\}$ on every double generalized eigenspace. Indeed, if $I_k\neq0$, then $\ker(A^*+k^2)$ is spanned by $\Psi_{2,k}$, whose observation $\varphi_{k\vert\omega}$ never vanishes; if $I_k = 0$, then $\ker(A^*+k^2)$ is spanned by $\Psi_{1,k}$ and $\Psi_{2,k}$, and, by \eqref{observedbasis}, the Fattorini--Hautus condition reduces to $\theta_{k\vert\omega}\neq0$. In both cases it is equivalent to $\rho_k>0$. Condition \eqref{FH} is therefore equivalent to approximate controllability \cite{AKBGBM_JMPA2019,BoyerOlive_MCRF2014,Fattorini_SICON1966}. It does not require $I_k\neq0$: the restriction of $\theta_k$ to $\omega$ may distinguish the missing direction even when the algebraic Jordan coefficient vanishes.

The elementary bounds \eqref{alphabounds} and \eqref{rhok} imply
$$
\vert I_k\vert \leq \frac{\rho_k}{\sqrt{m_\omega}},\qquad \rho_k \leq C.
$$
The first estimate controls every finite-time triangular correction uniformly. The second one implies that the number $T^\star$ in \eqref{Tstarintro} belongs to $[0, +\infty]$.

\subsection{Operator-range framework and range duality}\label{sec_range_duality}
Given any $\sigma\in\R$, let $W_\sigma$ be the space of all formal series $y = \sum_k(c_{1,k}\Phi_{1,k} + c_{2,k}\Phi_{2,k})$ such that
\begin{equation}\label{Wsigma}
\Vert y\Vert_{W_\sigma}^2 = \sum_{k \geq 1}e^{2\sigma k^2}(\vert c_{1,k}\vert^2 + \vert c_{2,k}\vert^2) < +\infty.
\end{equation} 
For $\sigma \geq 0$, $W_\sigma$ is a dense subspace of $\mathcal{H} = L^2(0,\pi;\R^2)$, with $W_0=\mathcal{H}$, and for $\sigma<0$, $W_\sigma$ is the completion of $\mathcal{H}$
for the norm \eqref{Wsigma}. The Riesz basis property and \eqref{semigroupblocks} show that, for $t\geq 0$, $e^{tA}$ maps $W_\sigma$ continuously into $W_{\sigma + t}$. The triangular transformation of the coefficients in \eqref{semigroupblocks}, as well as its inverse, is uniformly bounded at every fixed time. Consequently, $e^{TA}:W_{-T}\to\mathcal{H}$ is an isomorphism. For $\sigma \geq 0$, one has
\begin{equation}\label{semigrouprange}
W_\sigma = \Ran(e^{\sigma A}),
\end{equation}
with equivalence between \eqref{Wsigma} and the quotient norm on the range space.

Recall the controllability map $L_T$ from \eqref{globalgramian}, and set
$$
U_T=L^2(0,T;L^2(\omega)),
\qquad
\mathcal{R}_T=\Ran L_T.
$$ 
The terminal reachable space $\mathcal{R}_T$, endowed with the operator-range norm
\begin{equation}\label{RTnorm}
\Vert f\Vert_{\mathcal{R}_T} = \inf\{\Vert u\Vert_{U_T}\ \mid\  L_Tu=f\},
\end{equation}
is a Hilbert space isometric to $U_T/\ker L_T$, or equivalently to $(\ker L_T)^\perp$ through the minimum-norm preimage. This classical operator-range structure is recalled in \cite[(2.9) and Remark~2.1]{Tucsnak_ICM2022} and \cite[(2.6) and Remark~2.1]{ErvedozaLeBalchTucsnak_JFA2022}. The range-inclusion and duality principles used below are classical as well (see \cite{DoleckiRussell_SICON1977,Douglas_PAMS1966}, see also \cite[Proposition~12.1.2]{TucsnakWeiss_book2009}).

We now transport this Hilbert structure to initial data through the isomorphism $e^{TA}:W_{-T}\to\mathcal{H}$. We define the space of null-controllable initial data by
$$
\mathcal{C}_T
=\{y^0\in W_{-T}\ \mid\ -e^{TA}y^0\in\mathcal{R}_T\}.
$$
We equip it with the norm
\begin{equation}\label{CTnorm}
\Vert y^0\Vert_{\mathcal{C}_T}=\Vert e^{TA}y^0\Vert_{\mathcal{R}_T}=\inf\{\Vert u\Vert_{U_T}\ \mid\ e^{TA}y^0+L_Tu=0\}.
\end{equation}
The restriction $(e^{TA})_{\vert\mathcal{C}_T}:\mathcal{C}_T\rightarrow\mathcal{R}_T$ is then an onto isometry. Consequently, $\mathcal{C}_T$ is a Hilbert space continuously embedded in $W_{-T}$. Equivalently, $\mathcal{C}_T$ is the range of the bounded map $-(e^{TA})^{-1}L_T:U_T\to W_{-T}$, endowed with its operator-range norm \eqref{CTnorm}; here $(e^{TA})^{-1}$ denotes the inverse from $\mathcal{H}$ to $W_{-T}$.


\begin{remark}
By \eqref{CTnorm}, given any Hilbert space $X$ continuously embedded in $W_{-T}$, the inclusion $X\hookrightarrow\mathcal{C}_T$ is continuous if and only if there exists $C_{X,T}>0$ such that, for every $y^0\in X$, one can find $u\in U_T$ satisfying $e^{TA}y^0+L_Tu=0$ and $\Vert u\Vert_{U_T}\leq C_{X,T}\Vert y^0\Vert_X$.
In this sense, $\mathcal{C}_T$ is the maximal Hilbert space of null-controllable data in $W_{-T}$ with a uniform $L^2$ cost. For every $y^0\in W_{-T}$ and $u\in U_T$, the corresponding mild trajectory belongs to $C([0,T],W_{-T})$, while its terminal state $e^{TA}y^0+L_Tu$ belongs to $\mathcal{H}$. Thus the terminal constraint above is an identity in $\mathcal{H}$. The choice of the ambient space $W_{-T}$ loses no null-controllable datum of the heat scale: since $e^{TA}$ maps $W_\sigma$ isomorphically onto $W_{\sigma+T}$ for every $\sigma\in\R$ and $L_Tu\in\mathcal{H}$ for every $u\in U_T$, a datum $y^0\in W_\sigma$ with $\sigma<-T$ satisfies $e^{TA}y^0+L_Tu = 0$ for some control only if $e^{TA}y^0\in\mathcal{H}$, i.e., only if $y^0\in W_{-T}$.
\end{remark}

\begin{remark}[Dual norms]
Under \eqref{FH}, approximate controllability gives $\ker L_T^*=\{0\}$, $\mathcal{R}_T$ is dense in $\mathcal{H}$, and one has the pivot triple $\mathcal{R}_T\hookrightarrow\mathcal{H}\simeq\mathcal{H}'
\hookrightarrow\mathcal{R}_T'$ (continuous dense inclusions, with $\mathcal{H}$ as a pivot space). The continuous dual $\mathcal{R}_T'$ of $\mathcal{R}_T$ is canonically isometric to the completion of $\mathcal{H}$ for the dual norm $\Vert z\Vert_{\mathcal{R}_T'}=\Vert L_T^*z\Vert_{U_T}=\langle G_Tz,z\rangle_{\mathcal{H}}^{1/2}$ (see \cite[Definition~5.2 and Proposition~5.8]{ErvedozaLeBalchTucsnak_JFA2022}). Here $z\in\mathcal{H}$ represents the functional $\mathcal{R}_T\ni f\mapsto\langle z,f\rangle_{\mathcal{H}}$. We write $\langle f,\ell\rangle_{\mathcal{R}_T,\mathcal{R}_T'}=\ell(f)$; this pairing agrees with $\langle f,z\rangle_{\mathcal{H}}$ when $\ell$ is represented by $z\in\mathcal{H}$.
\end{remark}

\begin{remark}
When equipping $W_{-T}$ with the equivalent norm $\Vert e^{TA}y\Vert_{\mathcal{H}}$, the map $e^{TA}:W_{-T}\to\mathcal{H}$ is an onto isometry. Under \eqref{FH}, the equality $e^{TA}(\mathcal{C}_T)=\mathcal{R}_T$ also shows that $\mathcal{C}_T$ is dense in $W_{-T}$, and hence $\mathcal{C}_T\hookrightarrow W_{-T}\hookrightarrow\mathcal{C}_T'$ is another pivot triple, if one chooses $W_{-T}$ as a pivot space. 
\end{remark}

\begin{lemma}\label{rangeHUM}
Assume that $\ker L_T^*=\{0\}$.
For $y^0\in W_{-T}$, one has
\begin{equation}\label{rangecriterionCT}
y^0\in\mathcal{C}_T
\quad\Leftrightarrow\quad
\sup_{z\in\mathcal{H}\setminus\{0\}}
\frac{\vert\langle e^{TA}y^0,z\rangle_{\mathcal{H}}\vert}{\Vert L_T^*z\Vert_{U_T}}<+\infty.
\end{equation}
If $y^0\in\mathcal{C}_T$, then
\begin{equation}\label{dualCT}
\Vert y^0\Vert_{\mathcal{C}_T}
=\sup_{z\in\mathcal{H}\setminus\{0\}}\frac{\vert\langle e^{TA}y^0,z\rangle_{\mathcal{H}}\vert}{\Vert L_T^*z\Vert_{U_T}}
=\sup_{z\in\mathcal{H}\setminus\{0\}}\frac{\vert\langle e^{TA}y^0,z\rangle_{\mathcal{H}}\vert}
{\langle G_Tz,z\rangle_{\mathcal{H}}^{1/2}}. 
\end{equation}
\end{lemma}

\begin{proof}
Suppose first that $y^0\in\mathcal{C}_T$, and let $\widehat u\in(\ker L_T)^\perp$ be the minimum-norm control. Since $L_T\widehat u=-e^{TA}y^0$, one has
\begin{equation}\label{controlpairing}
\langle e^{TA}y^0,z\rangle_{\mathcal{H}} = -\langle\widehat u,L_T^*z\rangle_{U_T}
\qquad\forall z\in\mathcal{H}.
\end{equation}
The Cauchy--Schwarz inequality bounds the supremum in \eqref{dualCT} by $\Vert\widehat u\Vert_{U_T}$. If $\widehat u=0$, both sides vanish. If $\widehat u\neq0$, then $\overline{\Ran L_T^*}=(\ker L_T)^\perp$, so there are $z_n\in\mathcal{H}$ such that $L_T^*z_n\to\widehat u$ in $U_T$. For all sufficiently large $n$, one has $\Vert L_T^*z_n\Vert_{U_T}\neq 0$, and substitution in \eqref{controlpairing} shows that the corresponding quotients converge to $\Vert\widehat u\Vert_{U_T}$. This proves \eqref{dualCT}.

Conversely, suppose that the supremum in \eqref{rangecriterionCT} is finite. The map $L_T^*z\mapsto-\langle e^{TA}y^0,z\rangle_{\mathcal{H}}$, well defined because $L_T^*$ is injective, is a bounded linear functional on $\Ran L_T^*$. It extends to $(\ker L_T)^\perp$, and the Riesz representation theorem provides $\widehat u\in(\ker L_T)^\perp$ such that \eqref{controlpairing} holds. Hence $L_T\widehat u=-e^{TA}y^0$, so $y^0\in\mathcal{C}_T$; orthogonality to $\ker L_T$ makes $\widehat u$ the minimum-norm control.
\end{proof}

\begin{remark}[HUM realization]
Since $\mathcal{R}_T'$ is the completion of $\mathcal{H}$ for the norm $z\mapsto\Vert L_T^*z\Vert_{U_T}$, the operator $L_T^*$ extends uniquely to an onto isometry $\widetilde L_T^*:\mathcal{R}_T'\rightarrow(\ker L_T)^\perp$. The extended Gramian $\widetilde G_T=L_T\widetilde L_T^*:\mathcal{R}_T'\rightarrow\mathcal{R}_T$ is an onto isometry extending $G_T$ and satisfies
$\langle\widetilde G_Tz,w\rangle_{\mathcal{R}_T,\mathcal{R}_T'}=\langle\widetilde L_T^*z,\widetilde L_T^*w\rangle_{U_T}$ for all $z,w\in\mathcal{R}_T'$.

Let $y^0\in\mathcal{C}_T$. There is a unique $\widehat z\in\mathcal{R}_T'$ such that $\widetilde G_T\widehat z=-e^{TA}y^0$ in $\mathcal{R}_T$. Equivalently, $\langle\widetilde L_T^*\widehat z,L_T^*w\rangle_{U_T}
=-\langle e^{TA}y^0,w\rangle_{\mathcal{H}}$ for every $w\in\mathcal{H}$.
The vector $\widehat z$ is the unique minimizer on $\mathcal{R}_T'$ of the HUM functional $J_{T,y^0}(z)=\frac{1}{2}\Vert z\Vert_{\mathcal{R}_T'}^2+\langle e^{TA}y^0,z\rangle_{\mathcal{R}_T,\mathcal{R}_T'}$. The control $\widehat u=\widetilde L_T^*\widehat z$ is the unique minimum-norm null control, and
$$
\Vert y^0\Vert_{\mathcal{C}_T}
=\Vert\widehat u\Vert_{U_T}
=\Vert\widetilde L_T^*\widehat z\Vert_{U_T}
=\Vert\widehat z\Vert_{\mathcal{R}_T'}
=\big\langle\widetilde G_T\widehat z,\widehat z\big\rangle_{\mathcal{R}_T,\mathcal{R}_T'}^{1/2}.
$$
If $y^0\neq0$, the quotient in \eqref{dualCT}, after extension to $\mathcal{R}_T'$, is attained at $\widehat z$; it need not be attained in $\mathcal{H}$. For the HUM and Gramian constructions, see \cite[Sections~5.2.2--5.2.3]{Trelat_book2024}.
\end{remark}

\subsection{The spectral scale and the main theorem}\label{scaletheorem2}
Given any $\tau\in\R$, let $\mathcal{V}_\tau$ be the Hilbert space of all formal series $y=\sum_{k\geq1}(c_{1,k}\Phi_{1,k}+c_{2,k}\Phi_{2,k})$ such that
$$
\Vert y\Vert_{\mathcal{V}_\tau}^2=\sum_{k\geq1}e^{-2\tau k^2}\left(\frac{\vert c_{1,k}\vert^2}{\rho_k^2}+\vert c_{2,k}\vert^2\right)<+\infty.
$$
The scale is increasing in $\tau$. In the definition of the norm, the weight $e^{-2\tau k^2}$ records heat dissipation, while $\rho_k^{-2}$ compensates for the weak observation of the first direction in the $k$-th double eigenspace. For every $\tau\in\R$, the embedding $\mathcal{V}_\tau\hookrightarrow W_{-\tau}$ is continuous and dense: continuity follows from the uniform upper bound on $\rho_k$, and density follows from the fact that, since $\rho_k>0$ for every $k$, $\mathcal{V}_\tau$ contains all finite spectral sums. In particular, $\mathcal{V}_T\hookrightarrow W_{-T}$ continuously and densely.

\begin{theorem}\label{main}
Assume \eqref{FH}. There exists $\kappa_2>0$ such that, for every $T>0$, there exist $c_T>0$ and $C_T>0$ such that
\begin{equation}\label{refinedbounds}
c_T\sum_{k\geq1}k e^{-2Tk^2}\left(\frac{\vert c_{1,k}\vert^2}{\rho_k^2}+\vert c_{2,k}\vert^2\right)
\leq \Vert y\Vert_{\mathcal{C}_T}^2
\leq C_T\sum_{k\geq1}e^{-2Tk^2+\kappa_2k}\left(\frac{\vert c_{1,k}\vert^2}{\rho_k^2}+\vert c_{2,k}\vert^2\right)
\end{equation}
for every finite spectral sum $y=\sum_k(c_{1,k}\Phi_{1,k}+c_{2,k}\Phi_{2,k})$. The right-hand estimate extends to every series for which its right-hand side is finite, and the left-hand estimate holds for every $y\in\mathcal{C}_T$. 
Consequently, for every $\varepsilon>0$, one has the continuous inclusions
$$
\mathcal{V}_{T-\varepsilon}\hookrightarrow\mathcal{C}_T\hookrightarrow\mathcal{V}_T.
$$
\end{theorem}

The constant $\kappa_2$ is the one provided by Proposition \ref{twosidedmoment} of Appendix \ref{appendixmoments} for $N=2$: it comes from the biorthogonal family associated with the spectrum $\{k^2\mid k\geq1\}$ and the generalized exponentials $se^{-k^2s}$, and does not depend on $T$, $\beta$ or $\omega$.

Theorem \ref{main} (proved in Section \ref{proofmain}) gives an explicit spectral description of the exact space $\mathcal{C}_T$, almost sharp at the Gaussian scale. Indeed, the two estimates in \eqref{refinedbounds} contain the same factors $e^{-2Tk^2}$ and $\rho_k^{-2}$. Their only discrepancy is between the polynomial factor $k$ in the necessary estimate and the linear exponential $e^{\kappa_2k}$ in the sufficient one. Since $k\leq C_{\varepsilon}e^{\varepsilon k^2}$ and $e^{\kappa_2k}\leq C_{\varepsilon}e^{2\varepsilon k^2}$, these factors do not change the leading coefficient of $k^2$ in the exponent. This is why the sufficient condition uses $\mathcal{V}_{T-\varepsilon}$, whereas the necessary condition gives $\mathcal{C}_T\hookrightarrow\mathcal{V}_T$.

\begin{remark}[A sharper one-block estimate]\label{oneblockestimate}
For data supported in a single spectral block, the factor $k$ in the lower bound \eqref{refinedbounds} improves to $k^2$. More precisely, for every $T > 0$, there exists $c_T' > 0$, independent of $k$, such that
\begin{equation}\label{singleblocknecessary}
\Vert c_1\Phi_{1,k} + c_2\Phi_{2,k}\Vert_{\mathcal{C}_T}^2 \geq c_T' k^2 e^{-2Tk^2}\left(\frac{\vert c_1\vert^2}{\rho_k^2} + \vert c_2\vert^2\right)
\end{equation}
for all $k \geq 1$ and $(c_1,c_2) \in \R^2$.

Indeed, for $z_k = a_1\Psi_{1,k} + a_2\Psi_{2,k}$, apply \eqref{singlemoment} from Proposition \ref{twosidedmoment}, with $N = 2$, $X = L^2(\omega)$, $A_0 = (a_1\theta_k + a_2\varphi_k)_{\vert\omega}$ and $A_1 = (I_k a_1\varphi_k)_{\vert\omega}$. By \eqref{observedbasis} and \eqref{rhok}, this gives
$\langle G_T z_k,z_k\rangle_{\mathcal{H}} \leq \frac{C}{k^2}\big(\rho_k^2\vert a_1\vert^2 + \alpha_k\vert a_2\vert^2\big)$.
Restricting \eqref{dualCT} to these test vectors and optimizing over $(a_1,a_2)$ yields \eqref{singleblocknecessary}, by biorthogonality, \eqref{semigroupblocks}, and the weighted norm equivalence \eqref{horizonequiv} proved below; the constants $C$ and $\kappa_T$ involved there do not depend on $k$. Summing \eqref{singleblocknecessary} over $k$ does not improve the lower bound in \eqref{refinedbounds}, because the spectral blocks are not orthogonal in $\mathcal{C}_T$: the squared norm of a sum is not the sum of the squared norms of its blocks. The weaker factor $k$ in \eqref{refinedbounds} is the price of the interaction between distinct frequencies in the Gramian, which only the upper bound in \eqref{twosidedmomentineq} controls. This estimate will be used to prove noninclusions at subcritical times.
\end{remark}

\subsection{Comparison with the energy space and its critical time}

\begin{proposition}\label{Hcomparison}
Let $\tau\in\R$, and let $T^\star\in[0,+\infty]$ be defined by \eqref{Tstarintro}.
\begin{enumerate}
\item If $\tau<0$, then $\mathcal{V}_\tau\hookrightarrow\mathcal{H}$ continuously and properly.
\item One has $\mathcal{V}_0\hookrightarrow\mathcal{H}$. Moreover, $\mathcal{H} = \mathcal{V}_0$ with equivalent norms if and only if $\inf_k\rho_k>0$.
\item If $\tau>0$, then $\mathcal{V}_\tau$ is never contained in $\mathcal{H}$, and
\begin{equation}\label{HintoV}
\mathcal{H}\hookrightarrow\mathcal{V}_\tau\ \Leftrightarrow\ \sup_{k \geq 1}\frac{e^{-\tau k^2}}{\rho_k}< +\infty.
\end{equation}
Consequently, $\mathcal{H}$ is properly contained in $\mathcal{V}_\tau$ if $\tau>T^\star$, while $\mathcal{H}$ and $\mathcal{V}_\tau$ are incomparable if $0<\tau<T^\star$. When $0<T^\star<+\infty$ and $\tau=T^\star$, the answer depends on the subquadratic behavior of $\rho_k$.
\end{enumerate}
For the exact null-controllability space, $\mathcal{C}_T$ is not contained in $\mathcal{H}$ for any $T>0$. If $T>T^\star$, then $\mathcal{H}\hookrightarrow\mathcal{C}_T$ continuously and properly. If $T<T^\star$, then $\mathcal{H}$ and $\mathcal{C}_T$ are incomparable.
\end{proposition}

When $0<T^\star<+\infty$, nothing is asserted at $T=T^\star$: the value of $T^\star$ alone does not determine whether $\mathcal{H}\hookrightarrow\mathcal{C}_{T^\star}$. By \eqref{refinedbounds}, this inclusion holds if $\rho_k\geq c\,e^{-T^\star k^2+\kappa_2k/2}$ for every $k\geq1$; by \eqref{singleblocknecessary}, it fails if $\rho_k\leq Ce^{-T^\star k^2}$ along a subsequence. Between these two regimes, which differ by subquadratic factors only, Theorem \ref{main} gives no answer.

\begin{proof}
The Riesz basis identifies $\mathcal{H}$ with unweighted $\ell^2(\R^2)$. The first two items and \eqref{HintoV} follow by comparing the diagonal weights $e^{-2\tau k^2}\rho_k^{-2}$ and $e^{-2\tau k^2}$ with $1$. For $\tau<0$, strictness follows by taking $c_{1,k}=0$ and $c_{2,k}=1/k$, which defines an element of $\mathcal{H}\setminus\mathcal{V}_\tau$. If $\tau>0$, the datum $c_{1,k} = 0$, $c_{2,k} = e^{\tau k^2}/k$ belongs to $\mathcal{V}_\tau$ but not to $\mathcal{H}$. Formula \eqref{Tstarintro} gives the strict alternatives around $T^\star$.

The datum defined by $c_{1,k}=0$ and
$c_{2,k}=e^{Tk^2-\kappa_2k}/k$ makes the rightmost series in \eqref{refinedbounds} finite, but does not belong to $\mathcal{H}$. Hence $\mathcal{C}_T$ is never contained in $\mathcal{H}$. If $T>T^\star$, choose $\tau$ with $T^\star<\tau<T$; then $\mathcal{H}\hookrightarrow\mathcal{V}_\tau\hookrightarrow\mathcal{C}_T$. Suppose now that $T<T^\star$ and that $\mathcal{H}\subset\mathcal{C}_T$ as sets. The inclusion has a closed graph because both spaces embed continuously in $W_{-T}$, and is therefore continuous. Applying \eqref{singleblocknecessary} with $(c_1,c_2)=(1,0)$ would give $\rho_k\geq ck e^{-Tk^2}$, and hence $T^\star\leq T$. Thus $\mathcal{H}$ is not contained in $\mathcal{C}_T$.
\end{proof}

In the notation of the introduction, Proposition \ref{Hcomparison} gives $T_{\mathcal{H}}=T^\star$, i.e., $T^\star$ is the transition threshold at which the entire energy space enters the increasing family $(\mathcal{C}_T)_{T>0}$.

This is consistent with the datum-wise formulas of \cite{BBM_AHL2020,BoyerMorancey_CRM2023}. Indeed, set $T_{\min}(y)=\inf\{T>0\, \mid\, y\in\mathcal{C}_T\}$, the minimal null-control time of a formal series $y=\sum_k(c_{1,k}\Phi_{1,k}+c_{2,k}\Phi_{2,k})$. Theorem \ref{main} gives
$$
T_{\min}(y)=\max\Big(0,\ \limsup_{k\to +\infty}\frac{1}{2k^2}\log\Big(\frac{\vert c_{1,k}\vert^2}{\rho_k^2}+\vert c_{2,k}\vert^2\Big)\Big),
$$
with the convention $\log0 = -\infty$. Indeed, if $T$ exceeds the right-hand side, the sufficient series in \eqref{refinedbounds} converges, since $e^{\kappa_2k}$ is negligible with respect to any Gaussian gain; if $T$ is smaller than the right-hand side, the general term of the necessary series in \eqref{refinedbounds} is unbounded along a subsequence.
Moreover,
\begin{equation*}
T^\star=T_{\mathcal{H}}=\sup_{y\in\mathcal{H}}T_{\min}(y).
\end{equation*}
Indeed, $\mathcal{H}\hookrightarrow\mathcal{C}_T$ for every $T>T^\star$, so $\sup_{y\in\mathcal{H}}T_{\min}(y)\leq T^\star$. Conversely, if this supremum were smaller than $T^\star$, one could choose $T$ strictly between them. Monotonicity of $(\mathcal{C}_T)_{T>0}$ would then give $\mathcal{H}\subset\mathcal{C}_T$; the closed graph theorem would make this inclusion continuous, contradicting $T_{\mathcal{H}}=T^\star$. Thus the known datum-wise theory recovers $T^\star$ by taking the worst initial datum, while the present result describes the whole Hilbert space $\mathcal{C}_T$ at each fixed time.

For $T<T^\star$, the null-controllability space is not simply a smaller subspace of $L^2$: it contains singular data in the directly controlled direction while excluding some square-integrable data in the weakly observed direction. This explains the incomparability in Proposition \ref{Hcomparison}. Remark \ref{nonstandard} in Subsection \ref{sec_3examples} proves, for the arithmetic example below, incomparability with every isotropic spectral Sobolev or sub-Gaussian Gevrey space defined in the adapted basis. Crossing $T^\star$ therefore changes the nature of the controllable-data space, and not only its size.

\subsection{Comparison with the heat scale}
The weights defining $W_\sigma$ in \eqref{Wsigma} do not involve $\rho_k$, although the adapted basis depends on the coupling. For $\sigma\geq0$, this scale has the semigroup interpretation \eqref{semigrouprange}.

\begin{corollary}\label{Wphase}
Assume \eqref{FH}, let $T>0$, let $\sigma \geq -T$, and let $T^\star$ be defined by \eqref{Tstarintro}. If $T + \sigma>T^\star$, then $W_\sigma\hookrightarrow\mathcal{C}_T$ continuously. If $T + \sigma<T^\star$, then $W_\sigma$ is not contained in $\mathcal{C}_T$. No general conclusion is asserted when $T + \sigma = T^\star$.
\end{corollary}

\begin{proof}
If $T + \sigma>T^\star$, choose $\tau<T$ such that $\tau + \sigma>T^\star$. The definition of $T^\star$ gives $W_\sigma\hookrightarrow\mathcal{V}_\tau$, and Theorem \ref{main} gives $\mathcal{V}_\tau\hookrightarrow\mathcal{C}_T$. Conversely, suppose that $W_\sigma\subset\mathcal{C}_T$ as sets. Since $\sigma\geq-T$, both spaces embed continuously in $W_{-T}$; the closed graph theorem makes the inclusion continuous. Applying \eqref{singleblocknecessary} to $\Phi_{1,k}$ gives $\rho_k\geq ck e^{-(T+\sigma)k^2}$, hence $T^\star\leq T+\sigma$.
\end{proof}

At $\sigma = 0$, Corollary \ref{Wphase} recovers the $L^2$ minimal-time result of \cite{AKBGBdT_JMAA2016}. For $\sigma>0$, the data have undergone the coupled heat flow for time $\sigma$; for $-T \leq \sigma<0$, they are extrapolated data whose free evolution reaches $\mathcal{H}$ no later than the final time. The line $T+\sigma=T^\star$ therefore describes the trade-off between prior heat regularization and control time at the quadratic exponential scale. It does not replace Theorem \ref{main}: two couplings with the same $T^\star$ may have completely different sequences $\rho_k$ and hence different fixed-horizon spaces.

\subsection{Proof of Theorem \ref{main}}\label{proofmain}
By Lemma \ref{rangeHUM}, $\Vert y^0\Vert_{\mathcal{C}_T}$ is the supremum, over $z$, of the quotient of the pairing $\langle e^{TA}y^0,z\rangle_{\mathcal{H}}$ by $\langle G_Tz,z\rangle_{\mathcal{H}}^{1/2}$. The proof has three steps. Writing the adjoint observation in block coordinates turns $\langle G_Tz,z\rangle_{\mathcal{H}}$ into the integral of \eqref{twosidedmomentineq} with $N=2$, whose two bounds frame it by the weighted sums \eqref{gramiancomparison2}. The pairing is then computed in the same coordinates, the triangular correction due to the horizon $T$ being absorbed by \eqref{horizonequiv}. Finally, the upper bound in \eqref{refinedbounds} follows from the Cauchy--Schwarz inequality together with the lower bound in \eqref{gramiancomparison2}, applied to every test vector $z$, whereas the lower bound in \eqref{refinedbounds} follows from the upper bound in \eqref{gramiancomparison2}, applied to one well-chosen $z$.

\paragraph{The adjoint observation in block coordinates.}

By \eqref{Jordanrelations}, the restrictions of $A$ and $A^*$ to their generalized eigenspaces at $-k^2$ have, respectively, the following matrices in the ordered bases $(\Phi_{1,k},\Phi_{2,k})$ and $(\Psi_{1,k},\Psi_{2,k})$:
$$
[A]_{\Phi_k}=\begin{pmatrix}-k^2&I_k\\0&-k^2\end{pmatrix},\qquad [A^*]_{\Psi_k}=\begin{pmatrix}-k^2&0\\I_k&-k^2\end{pmatrix}.
$$
Their exponentials give the semigroup formulas \eqref{semigroupblocks}--\eqref{adjointblocks}.

Let $a=(a_k)_{k\geq1}$ be finitely supported, with $a_k=(a_{1,k},a_{2,k})\in\R^2$, and set
$$
z[a]=\sum_{k\geq1}\Psi_ka_k=\sum_{k\geq1}(a_{1,k}\Psi_{1,k}+a_{2,k}\Psi_{2,k}).
$$
Using \eqref{observedbasis} and \eqref{adjointblocks}, we obtain
\begin{equation}\label{globalobservation2}
B^*e^{sA^*}z[a]=\sum_{k\geq1}e^{-k^2s}(A_{0,k}+sA_{1,k}),
\end{equation}
where $A_{0,k}=(a_{1,k}\theta_k+a_{2,k}\varphi_k)_{\vert\omega}$ and $A_{1,k}=(I_ka_{1,k}\varphi_k)_{\vert\omega}$.
The orthogonality in \eqref{observedbasis} and the definition \eqref{rhok} give
\begin{equation}\label{coefficientidentity2}
\Vert A_{0,k}\Vert_{L^2(\omega)}^2+\Vert A_{1,k}\Vert_{L^2(\omega)}^2
= \rho_k^2\vert a_{1,k}\vert^2+\alpha_k\vert a_{2,k}\vert^2
= \langle \Gamma_ka_k,a_k\rangle_{\R^2}
= \big\Vert \Gamma_k^{1/2}a_k\big\Vert_{\R^2}^2,
\end{equation}
where the matrix
$$
\Gamma_k=\begin{pmatrix}\rho_k^2&0\\0&\alpha_k\end{pmatrix}
$$
represents the quadratic form given by the squared norms of the observed polynomial coefficients. 
This notation will also be used for finite cascades, where $\Gamma_k$ need not be diagonal. The coefficient identity \eqref{coefficientidentity2} corresponds to the double-block computation of \cite{BoyerMorancey_AMBP2025}, written in the present normalization.

Substituting \eqref{globalobservation2} into \eqref{gramianobservationidentity} and applying Proposition \ref{twosidedmoment} in Appendix \ref{appendix}, with $N=2$ and $X=L^2(\omega)$, gives, by \eqref{coefficientidentity2},
\begin{multline}\label{gramiancomparison2}
c_{T,2}\sum_{k\geq1}e^{-\kappa_2k}\big\Vert \Gamma_k^{1/2}a_k\big\Vert_{\R^2}^2
\leq\langle G_Tz[a],z[a]\rangle_{\mathcal{H}}\\
=\int_0^T\big\Vert B^*e^{sA^*}z[a]\big\Vert_{L^2(\omega)}^2\,ds
\leq C_{T,2}\sum_{k\geq1}\frac{1}{k}\big\Vert \Gamma_k^{1/2}a_k\big\Vert_{\R^2}^2
\end{multline}
with the constants of Proposition \ref{twosidedmoment} for $N=2$.
Since the dual Riesz basis is complete and the matrices $\Gamma_k^{1/2}$ are uniformly bounded, spectral truncation extends \eqref{gramiancomparison2} to every $a\in\ell^2(\R^2)$.

\paragraph{The numerator and duality argument.}
For $y^0\in W_{-T}$, write
$$
y^0=\sum_{k\geq1}(c_{1,k}\Phi_{1,k}+c_{2,k}\Phi_{2,k}),\qquad c_k=(c_{1,k},c_{2,k}).
$$
Formula \eqref{semigroupblocks} and biorthogonality give, for every finitely supported $a$,
\begin{equation}\label{numerator2}
\langle e^{TA}y^0,z[a]\rangle_{\mathcal{H}}=\sum_{k\geq1}e^{-Tk^2}\langle M_k(T)c_k,a_k\rangle_{\R^2},\qquad M_k(T)=\begin{pmatrix}1&TI_k\\0&1\end{pmatrix}.
\end{equation}
The triangular correction $M_k(T)$ preserves the weighted norm $c\mapsto\Vert \Gamma_k^{-1/2}c\Vert_{\R^2}$ up to constants independent of $k$. Indeed,
$$
\Gamma_k^{-1/2}M_k(T)\Gamma_k^{1/2}=\begin{pmatrix}1&TI_k\sqrt{\alpha_k}/\rho_k\\0&1\end{pmatrix},
$$
and $\vert I_k\vert\sqrt{\alpha_k}\leq\rho_k$. Hence, this conjugated matrix and its inverse, obtained by replacing $T$ with $-T$, are uniformly bounded in $k$. Together with \eqref{alphabounds}, this gives a constant $\kappa_T\geq1$ such that, for every $k\geq1$ and $c=(c_1,c_2)\in\R^2$,
\begin{equation}\label{horizonequiv}
\kappa_T^{-1}\left(\frac{\vert c_1\vert^2}{\rho_k^2}+\vert c_2\vert^2\right)
\leq\Vert \Gamma_k^{-1/2}M_k(T)c\Vert_{\R^2}^2
\leq\kappa_T\left(\frac{\vert c_1\vert^2}{\rho_k^2}+\vert c_2\vert^2\right).
\end{equation}

Assume first that the series on the right-hand side of \eqref{refinedbounds} is finite. Since $\rho_k$ is uniformly bounded, the corresponding initial datum belongs to $W_{-T}$. Writing each summand of \eqref{numerator2} as
$$
e^{-Tk^2}\langle M_k(T)c_k,a_k\rangle_{\R^2}=\big\langle e^{-Tk^2+\kappa_2k/2}\Gamma_k^{-1/2}M_k(T)c_k,e^{-\kappa_2k/2}\Gamma_k^{1/2}a_k\big\rangle_{\R^2},
$$
the Cauchy--Schwarz inequality in $\ell^2(\R^2)$ gives
$$
\vert\langle e^{TA}y^0,z[a]\rangle_{\mathcal{H}}\vert^2\leq\Big(\sum_{k\geq1}e^{-2Tk^2+\kappa_2k}\Vert \Gamma_k^{-1/2}M_k(T)c_k\Vert_{\R^2}^2\Big)\Big(\sum_{k\geq1}e^{-\kappa_2k}\big\Vert \Gamma_k^{1/2}a_k\big\Vert_{\R^2}^2\Big).
$$
The first factor is finite by \eqref{horizonequiv}, while the lower bound in \eqref{gramiancomparison2} controls the second factor by $c_{T,2}^{-1}$$\langle G_Tz[a],z[a]\rangle_{\mathcal{H}}$. The quotient in \eqref{rangecriterionCT} is thus uniformly bounded for finite spectral test vectors. By density and continuity of the numerator and of $L_T^*$, the bound extends to every $z\in\mathcal{H}$. Lemma \ref{rangeHUM} yields $y^0\in\mathcal{C}_T$ and
$$
\Vert y^0\Vert_{\mathcal{C}_T}^2\leq c_{T,2}^{-1}\sum_{k\geq1}e^{-2Tk^2+\kappa_2k}\Vert \Gamma_k^{-1/2}M_k(T)c_k\Vert_{\R^2}^2.
$$
Using \eqref{horizonequiv} proves the right-hand inequality in \eqref{refinedbounds}, with $C_T = \kappa_Tc_{T,2}^{-1}$.

For the left-hand inequality, let $y^0\in\mathcal{C}_T$ and let $F\subset\N\setminus\{0\}$ be finite.
Choose $a_k=k e^{-Tk^2}\Gamma_k^{-1}M_k(T)c_k$ for $k\in F$ and $a_k=0$ otherwise, and set $z_F=z[a]$. Then \eqref{numerator2} gives
$$
\langle e^{TA}y^0,z_F\rangle_{\mathcal{H}}=\sum_{k\in F}k e^{-2Tk^2}\Vert \Gamma_k^{-1/2}M_k(T)c_k\Vert_{\R^2}^2=S_F,
$$
while the upper bound in \eqref{gramiancomparison2} gives $\langle G_Tz_F,z_F\rangle_{\mathcal{H}}\leq C_{T,2} S_F$.
The dual formula \eqref{dualCT} therefore implies $S_F^2\leq C_{T,2}\Vert y^0\Vert_{\mathcal{C}_T}^2 S_F$.
Dividing by $S_F$ when $S_F>0$, and noting that the resulting inequality also holds when $S_F=0$, we obtain, using \eqref{horizonequiv},
$$
\Vert y^0\Vert_{\mathcal{C}_T}^2\geq C_{T,2}^{-1}S_F\geq(C_{T,2}\kappa_T)^{-1}\sum_{k\in F}k e^{-2Tk^2}\left(\frac{\vert c_{1,k}\vert^2}{\rho_k^2}+\vert c_{2,k}\vert^2\right).
$$
Taking $F=\{1,\ldots,n\}$ and letting $n\to+\infty$ proves the left-hand inequality in \eqref{refinedbounds}, with $c_T = (C_{T,2}\kappa_T)^{-1}$. The inclusion $\mathcal{C}_T\hookrightarrow\mathcal{V}_T$ follows by discarding the factor $k\geq1$. Finally, the bound $e^{\kappa_2k}\leq C_{\varepsilon}e^{2\varepsilon k^2}$ in the right-hand inequality of \eqref{refinedbounds} gives $\mathcal{V}_{T-\varepsilon}\hookrightarrow\mathcal{C}_T$.

\begin{remark}[The subquadratic loss and the origin of the threshold]
The loss in the M\"untz--Sz\'asz inequality is linear in $k$, that is, in the square root of the eigenvalue $k^2$, whereas heat dissipation is quadratic in $k$. Thus $e^{\kappa_2k}\leq C_{\varepsilon}e^{2\varepsilon k^2}$ for every $\varepsilon>0$, which explains the shift from $T$ to $T-\varepsilon$ in the sufficient inclusion. This linear exponential loss is negligible at the Gaussian scale used to define the critical times.

The lower inequality in \eqref{twosidedmomentineq} holds for every $T>0$ and every fixed multiplicity $N$; only its constant deteriorates as $T\to0^+$. The exponential fast-control scale is not an artifact of a particular proof: in a broad selfadjoint heat-semigroup setting, \cite{BoulardHayatTrelat_arXiv2026} proves an optimal geometric lower barrier for weighted observability and recovers the maximal-distance lower bound for fast controls. There is nevertheless no positive threshold in the M\"untz--Sz\'asz inequality itself. This differs from the classical Ingham inequality for separated real frequencies, whose coercive estimate requires an interval longer than a gap-dependent critical length \cite{Ingham_MZ1936}; see also the moment problem with moving control in \cite{MartinRosierRouchon_SICON2013}. Here the critical times are created by the degeneration of the observed generalized eigenspaces, measured by $\rho_k$, not by the family $(e^{-k^2t},te^{-k^2t})_{k\geq1}$ alone.
\end{remark}

\subsection{Identification with the moments of the disjoint geometry}\label{spectralphenomena}
Assume \eqref{disjoint}, so that $\supp(\beta)\subset[0,a]\cup[b,\pi]$, and set
$$
I_{1,k} = \int_0^a\beta(x)\varphi_k(x)^2\, dx.
$$
The next proposition is the present normalization of the geometric reduction proved in \cite[Theorem 1.4 and Proposition 2.1]{BoyerMorancey_AMBP2025}, building on the minimal-time formula announced in \cite{AKBGBdT_CRM2014} and developed in \cite{AKBGBdT_JMAA2016}.
We include a direct proof to track the polynomial factors used in the examples; this proposition does not require \eqref{FH}.

\begin{proposition}\label{geometricrho}
There are constants $C \geq c>0$ such that
\begin{equation}\label{rhoequivalence}
\frac{c}{k}\max(\vert I_k\vert,\vert I_{1,k}\vert) \leq \rho_k \leq C\max(\vert I_k\vert,\vert I_{1,k}\vert)
\end{equation}
for every $k \geq 1$. In particular, $\rho_k = 0$ if and only if $I_k = I_{1,k} = 0$, and
\begin{equation}\label{Tstargeometry}
T^\star = \limsup_{k\to +\infty}\frac{\min(-\log\vert I_k\vert,-\log\vert I_{1,k}\vert)}{k^2},
\end{equation}
with the convention $-\log0 = +\infty$. This is the distributed minimal time in \cite{AKBGBdT_JMAA2016}.
\end{proposition}

The intrinsic quantity $\rho_k$ and Theorem \ref{main} do not require disjoint supports. Proposition \ref{geometricrho} converts the geometry of the generalized eigenspaces into the explicit moments of the classical model. More general open control sets and Sturm--Liouville operators are treated through block Gram matrices in \cite{BoyerMorancey_AMBP2025}.

\begin{proof}
Set $\chi_k(x) = \sqrt{2/\pi}\cos(kx)$. On $\omega$, equation \eqref{psik} becomes $\psi_k''+k^2\psi_k = I_k\varphi_k$, hence
\begin{equation}\label{psionomega}
\psi_k(x) = A_k\chi_k(x) + B_k\varphi_k(x)-\frac{I_k}{2k}x\chi_k(x),\qquad x\in\omega.
\end{equation}
Consider $w_k = \psi_k\varphi_k'-\psi_k'\varphi_k$. Equation \eqref{psik} gives
$w_k'(x) = (\beta(x)-I_k)\varphi_k(x)^2$,
and the Dirichlet conditions give
$w_k(x) = \int_0^x\beta(s)\varphi_k(s)^2ds-I_k\int_0^x\varphi_k(s)^2ds$.
On $\omega$, using $\varphi_k' = k\chi_k$ and $\chi_k' = -k\varphi_k$, the same Wronskian computed from \eqref{psionomega} is
$$
w_k(x) = \frac{2k}{\pi}A_k-\frac{I_k}{\pi}\Big(x-\frac{\sin(2kx)}{2k}\Big),\qquad x\in\overline\omega.
$$
Since $\int_0^a\varphi_k(s)^2ds = \frac{a}{\pi}-\frac{\sin(2ka)}{2\pi k}$, the terms carrying $I_k$ in the two expressions of $w_k(a)$ are equal, and we obtain the exact identity
$A_k = \frac{\pi}{2k}I_{1,k}$.
The coefficient $B_k$ does not appear, since the Wronskian of $\varphi_k$ with itself vanishes. Thus the cosine part of the generalized eigenfunction on $\omega$ is exactly, and not only asymptotically, the coupling accumulated on the left of $\omega$.

By \eqref{deltatheta}, $\Vert\theta_k\Vert_{L^2(\omega)}$ is the distance from $\psi_k$ to $\Span(\varphi_k)$ in $L^2(\omega)$, so that \eqref{psionomega} gives
$$
\Vert\theta_k\Vert_{L^2(\omega)} = \dist_{L^2(\omega)}\Big(A_k\chi_k-\frac{I_k}{2k}x\chi_k,\Span(\varphi_k)\Big).
$$
For each $k\geq1$, the functions $\chi_k$, $x\chi_k$ and $\varphi_k$ are linearly independent in $L^2(\omega)$: applying $\partial_{xx}+k^2$ to a linear relation $\lambda\chi_k+\mu x\chi_k+r\varphi_k = 0$ on $\omega$ gives $-2k\mu\varphi_k = 0$, hence $\mu = 0$, and then $\lambda = r = 0$. Their Gram matrices in $L^2(\omega)$ are therefore positive definite. As $k\to+\infty$, these matrices converge to the positive definite matrix $\frac{1}{\pi}\mathrm{diag}(G,\vert\omega\vert)$, where $G$ is the Gram matrix of $(1,x)$ in $L^2(\omega)$. Thus there are $C\geq c>0$ such that, for every $k\geq1$ and all $\lambda,\mu\in\R$,
$$
c(\vert\lambda\vert + \vert\mu\vert) \leq \dist_{L^2(\omega)}(\lambda\chi_k + \mu x\chi_k,\Span(\varphi_k)) \leq C(\vert\lambda\vert + \vert\mu\vert).
$$
With $\lambda = A_k = \frac{\pi}{2k}I_{1,k}$ and $\mu = -I_k/(2k)$, this gives
$$
\frac{c}{k}(\vert I_{1,k}\vert + \vert I_k\vert) \leq \Vert\theta_k\Vert_{L^2(\omega)} \leq \frac{C}{k}(\vert I_{1,k}\vert + \vert I_k\vert).
$$
Since $\rho_k^2 = I_k^2\alpha_k + \Vert\theta_k\Vert_{L^2(\omega)}^2$ and $m_\omega \leq \alpha_k \leq 1$ by \eqref{alphabounds}, both bounds of \eqref{rhoequivalence} follow, the lower one from $\rho_k \geq \sqrt{m_\omega}\vert I_k\vert$ and $\rho_k \geq \Vert\theta_k\Vert_{L^2(\omega)} \geq c\vert I_{1,k}\vert/k$. These bounds also give the equivalence of the vanishing conditions, without assuming \eqref{FH}. Taking logarithms in \eqref{rhoequivalence} and dividing by $k^2$ proves \eqref{Tstargeometry}.
\end{proof}

By symmetry, one may replace $I_{1,k}$ with $I_{2,k} = \int_b^\pi\beta\varphi_k^2$ in \eqref{rhoequivalence}. The two moments record how the coupling accumulated on either side of the observation interval feeds the cosine part of the generalized eigenfunction. This is the source of the geometric dependence on the relative position of $\omega$ and $\supp(\beta)$.

\subsection{Three model situations}\label{sec_3examples}

\begin{example}
Assume \eqref{disjoint} and $\int_0^\pi\beta(x)\, dx\neq0$. The Riemann--Lebesgue lemma gives
$I_k\rightarrow\frac{1}{\pi}\int_0^\pi\beta(x)\, dx\neq0$.
Hence $\inf_{k \geq k_0}\rho_k>0$ and $T^\star = 0$. Under \eqref{FH}, the system is null controllable from $\mathcal{H}$ in every positive time. Corollary \ref{Wphase} says more: at time $T$, it is null controllable from every $W_\sigma$ with $\sigma>-T$. The borderline case $\sigma=-T$ is not decided by the quadratic critical-time formula.
\end{example}

\begin{example}
Fix $\vartheta\in[0, +\infty]$. Following the proof of \cite[Theorem 7.3]{AKBGBdT_JMAA2016}, choose an irrational number $\alpha\in(0,a/(3\pi))$ such that
$$
\limsup_{k\to +\infty}\frac{-\log\vert\sin(k\alpha\pi)\vert}{k^2} = \frac{\vartheta}{2}
$$
(such numbers are constructed there), a quadratic irrational $\ell\in(2a/(3\pi),4a/(3\pi))$, and set $a_1 = \ell/2-\alpha$, so that $a_1>0$ and $(a_1 + 2\alpha)\pi<a$. Consider the coupling supported on two adjacent intervals of equal length located to the left of $\omega$,
$$
\beta = 1_{[a_1\pi,(a_1 + \alpha)\pi)}-1_{[(a_1 + \alpha)\pi,(a_1 + 2\alpha)\pi]}.
$$
Since $\varphi_k^2 = \frac{1}{\pi}(1-\cos(2kx))$ and the two intervals have the same length, a direct computation gives
$$
I_{1,k} = I_k = -\frac{2}{k\pi}\sin^2(k\alpha\pi)\sin(2k\pi(a_1 + \alpha)) = -\frac{2}{k\pi}\sin^2(k\alpha\pi)\sin(k\pi\ell).
$$
Quadratic irrationals are badly approximable, so $\vert\sin(k\pi\ell)\vert\geq c/k$ for every $k\geq1$: the last sine never vanishes and contributes at most a logarithmic term to $-\log\vert I_k\vert$ on the whole sequence. As $\alpha$ is irrational, $I_k\neq0$ for every $k\geq1$, so that \eqref{FH} holds, and Proposition \ref{geometricrho} gives $T^\star=\vartheta$.

Because $\alpha$ is irrational, $\vert\sin(k\alpha\pi)\vert$ stays away from zero along an infinite subsequence. The polynomial lower bound for the other sine and Proposition \ref{geometricrho} then show that $\rho_k$ has only polynomial size along this subsequence. For $0<\vartheta<+\infty$, there is also a resonant subsequence on which $\rho_k=e^{-\vartheta k^2+\mathrm{o}(k^2)}$. The first-coordinate weight of $\mathcal{V}_T$ therefore has the respective asymptotic forms $e^{-2Tk^2+\mathrm{o}(k^2)}$ and $e^{2(\vartheta-T)k^2+\mathrm{o}(k^2)}$ on these two subsequences. When $0<T<\vartheta$, the regularity requirements differ exponentially between them. If $\vartheta=+\infty$, one instead obtains a subsequence on which $-\log\rho_k/k^2\to+\infty$, and there is no finite time of null controllability from $\mathcal{H}$. These arithmetic weights are analogous to the Diophantine controllable-data spaces in \cite{Zuazua_JEMS2011}. The next remark makes precise why isotropic spectral weights cannot describe this behavior.
\end{example}

\begin{remark}\label{nonstandard}
We compare $\mathcal{C}_T$ with isotropic spectral scales defined in the adapted Riesz basis. Assume $0<\vartheta<+\infty$ and $0<T<\vartheta$, and let $X$ be the Hilbert space of formal series with norm $\Vert y\Vert_X^2=\sum_kw_k(\vert c_{1,k}\vert^2+\vert c_{2,k}\vert^2)$, where $w_k>0$ and
$$
\gamma = \lim_{k\to +\infty}\frac{\log w_k}{2k^2}
$$
exists in $[-\infty,+\infty]$. Polynomial weights $w_k=k^{2s}$ define spectral Sobolev scales and have $\gamma=0$, as do the sub-Gaussian Gevrey weights $w_k=e^{2ck^{\nu}}$ with $c>0$ and $0<\nu<2$ (Gevrey index $1/\nu>\frac{1}{2}$). These are scales associated with the adapted basis; the Riesz basis property in $L^2$ alone does not identify them with the usual spatial Sobolev or Gevrey spaces. Any set inclusion between $X$ and $\mathcal{C}_T$ is continuous by the closed graph theorem, since convergence in either space implies convergence of every spectral coefficient.

If $X\hookrightarrow\mathcal{C}_T$, the one-block estimate \eqref{singleblocknecessary}, tested on the first direction, gives $w_k\geq ck^2e^{-2Tk^2}\rho_k^{-2}$.
Along the resonant subsequence, on which $\rho_k=e^{-\vartheta k^2+\mathrm{o}(k^2)}$, this implies $\gamma\geq\vartheta-T>0$. Conversely, if $\mathcal{C}_T\hookrightarrow X$, then Theorem \ref{main} gives $\mathcal{V}_\tau\hookrightarrow X$ for every $\tau<T$. Testing this inclusion on one block in the second direction gives $w_k\leq Ce^{-2\tau k^2}$, hence $\gamma\leq-T<0$. Thus no isotropic scale with a single Gaussian rate can be equivalent to $\mathcal{C}_T$. When $\gamma=0$, neither inclusion is possible: $\mathcal{C}_T$ is incomparable with every spectral Sobolev space and with every spectral Gevrey space in the range above. The weight $\rho_k^{-2}$ records this anisotropy, analogous to the coupling-dependent oscillatory weights of \cite{LPT_arXiv2026} and the point-actuated heat setting of \cite{Zuazua_JEMS2011}.
\end{remark}

\begin{example}
Let $\beta = 1_{(\pi/4,3\pi/8)}-1_{(5\pi/8,3\pi/4)}$.
The function is antisymmetric about $\pi/2$, whereas $\varphi_k^2$ is symmetric, so $I_k = 0$ for every $k$. If $3\pi/8<a<b<5\pi/8$, then
$I_{1,k} = \int_{\pi/4}^{3\pi/8}\varphi_k(x)^2\, dx = \frac{1}{8} + \mathrm{O}(\frac{1}{k})$.
Proposition \ref{geometricrho} gives $\rho_k \geq c/k$ and $T^\star = 0$. The system is approximately and null controllable although every algebraic Jordan coefficient vanishes. If instead $\omega\subset(0,\pi/4)$, then $I_k = I_{1,k} = 0$, so $\rho_k = 0$ and approximate controllability fails; this second placement lies outside hypothesis \eqref{FH}. This example, related to \cite[Proposition 2.5]{BoyerMorancey_AMBP2025}, shows why $I_k$ alone is not the correct invariant for distributed control.
\end{example}

The preceding examples also clarify the relation with \cite{LPT_arXiv2026}. In both the hyperbolic--parabolic and the present parabolic setting, a sequence of small modal couplings creates spaces whose weights can oscillate along selected subsequences. Here the left-hand estimate in \eqref{refinedbounds} proves that the nonstandard weight is unavoidable at the Gaussian scale.

\section{Finite strictly upper triangular heat cascades}\label{finitecascades}
For two equations, one number $\rho_k$ describes the potentially weak direction in the $k$-th spectral block. With $N\geq3$ components, the block observation map has $N$ singular values, and the corresponding families may have different quadratic exponential orders. The least observed direction determines the critical time of the full state space; the complete list yields an ordered family of modal transition times.

The analysis below allows all bounded strictly upper triangular couplings, not only couplings between neighboring equations. Such systems are still triangular cascades. We reserve the term \emph{nearest-neighbor chain} for the case in which only the entries $(j,j+1)$ are present. In the coupling graph, there is an arrow $j\to i$ when the entry $(i,j)$ is nonzero; a \emph{directed coupling path} is a sequence of such arrows carrying the influence of the controlled component towards an uncontrolled one. Null controllability for classes of three-equation parabolic systems with space-dependent coefficients and one or two localized controls was established by Carleman and decomposition arguments in \cite{Mauffrey_JMPA2013}. Boundary-coupled heat cascades have also been studied through cascade-level spectral reductions and finite-dimensional feedback designs \cite{LPT_Automatica2026,LPT_SCL2026,LPT_EJC2026}; a related Riesz-spectral output-feedback design for a wave--heat cascade is developed in \cite{LPT_TAC2026}. The latter four works concern stabilization rather than the fixed-horizon space of null-controllable initial data considered here.

We follow the order of Section \ref{twocascade}: spectral preliminaries and the fixed-horizon theorem, critical times and comparison with heat scales, then the proof and examples. Subsection \ref{proofmainN} repeats the two steps of Subsection \ref{proofmain}: a moment estimate for the full adjoint observation, followed by a weighted duality argument. The notation is kept unchanged wherever the same construction applies.

\subsection{Generator and adapted block bases}
Fix $N \geq 2$ and consider
$$
\partial_t y = Ay + Bu,\qquad A = \partial_{xx}\,\Id_{\R^N} + \sum_{1 \leq i<j \leq N}\beta_{ij}E_{i,j},\qquad Bu = 1_\omega u\, e_N,
$$
on $\mathcal{H}=L^2(0,\pi;\R^N)$, with $D(A)=H^2(0,\pi;\R^N)\cap H_0^1(0,\pi;\R^N)$. Here $(e_1,\ldots,e_N)$ is the canonical basis, $E_{i,j}$ is the corresponding matrix unit, and every $\beta_{ij}$ is real-valued and belongs to $L^\infty(0,\pi)$.
Since $N$ is fixed, we omit it from the notation for the operators and spaces: $A$, $B$, $\mathcal{H}$, $L_T$, $G_T$, $\mathcal{R}_T$ and $\mathcal{C}_T$ now refer to this cascade. Constants may depend on $N$; the index is retained in $\kappa_N$ to identify the multiplicity in Proposition \ref{twosidedmoment}. 

\begin{lemma}[Adapted block bases]\label{rieszN}
For every $k \geq 1$, the generalized eigenspaces $E_k$ of $A$ and $E_k^*$ of $A^*$ at $-k^2$ have dimension $N$. There are maps $\Phi_k:\R^N\to E_k$ and $\Psi_k:\R^N\to E_k^*$ such that $(\Phi_{j,k})_{j,k}$ and $(\Psi_{j,k})_{j,k}$, where $\Phi_{j,k}=\Phi_ke_j$ and $\Psi_{j,k}=\Psi_ke_j$, are dual Riesz bases of $\mathcal{H}$. They may be normalized so that
\begin{equation}\label{flagnormalizationN}
\Phi_{j,k} = \varphi_ke_j + \zeta_{j,k},\qquad \Psi_{j,k} = \varphi_ke_j + \xi_{j,k},
\end{equation}
where $\zeta_{j,k}$ has nonzero entries only in components $1,\ldots,j-1$, and $\xi_{j,k}$ only in components $j+1,\ldots,N$. In these bases,
\begin{equation}\label{similarityN}
A\Phi_k = \Phi_k(-k^2\Id_{\R^N}+J_k),\qquad A^*\Psi_k = \Psi_k(-k^2\Id_{\R^N}+J_k^*),
\end{equation}
where $J_k$ is strictly upper triangular, $J_k^N=0$ and $\sup_k\Vert J_k\Vert<+\infty$.
\end{lemma}

This is the counterpart of Lemma \ref{lem_Jordan}; its proof is given in Subsection \ref{app_proof_lem_rieszN} of the appendix. The matrices $J_k$ need not be single Jordan blocks: the construction retains the entire generalized eigenspace with uniformly conditioned coordinates. For $N=2$, the bases of Lemma \ref{lem_Jordan} are admissible and give $J_k=\begin{pmatrix}0&I_k\\0&0\end{pmatrix}$. Formula \eqref{similarityN} then reduces to \eqref{Jordanrelations}.

The entries of $J_k$ are effective modal coupling coefficients. Its first superdiagonal is given by
$$
(J_k)_{j,j+1} = \int_0^\pi\beta_{j,j+1}(x)\varphi_k(x)^2\,dx\qquad\forall j\in\{1,\ldots,N-1\},
$$
which generalizes \eqref{def_Ik}. Indeed, $(J_k)_{j,j+1}=\langle(A+k^2)\Phi_{j+1,k},\Psi_{j,k}\rangle_{\mathcal{H}}$ by biorthogonality. By \eqref{flagnormalizationN}, the components $j+1,\ldots,N$ of $(A+k^2)\Phi_{j+1,k}$ vanish, while the components $1,\ldots,j-1$ of $\Psi_{j,k}$ vanish; the pairing therefore involves only the $j$-th components, namely $(\partial_{xx}+k^2)(\zeta_{j+1,k})_j+\beta_{j,j+1}\varphi_k$ and $\varphi_k$. Since $(\zeta_{j+1,k})_j\in H^2(0,\pi)\cap H_0^1(0,\pi)$, the term $(\partial_{xx}+k^2)(\zeta_{j+1,k})_j$ is orthogonal to $\varphi_k$, and the formula follows. The higher superdiagonals may also contain contributions from directed coupling paths and from the spatial correctors.

The semigroup formulas corresponding to \eqref{semigroupblocks} and \eqref{adjointblocks} are
\begin{equation}\label{semigroupblocksN}
e^{tA}\Phi_kc = e^{-k^2t}\Phi_ke^{tJ_k}c,\qquad e^{tA^*}\Psi_ka = e^{-k^2t}\Psi_ke^{tJ_k^*}a.
\end{equation}

\subsection{The coefficient observation form}
Since $J_k^N=0$, the observation of an adjoint block is a polynomial heat mode:
\begin{equation}\label{blockobservationN}
B^*e^{sA^*}\Psi_ka = e^{-k^2s}\sum_{m=0}^{N-1}s^m C_{k,m}a,\qquad C_{k,m} = \frac{1}{m!}B^*\Psi_k(J_k^*)^m.
\end{equation}
Each $C_{k,m}$ maps $\R^N$ into $L^2(\omega)$. The sum of the squared norms of these coefficients defines a nonnegative quadratic form, represented by
\begin{equation}\label{coefficientmatrixN}
\mathcal{O}_ka = (C_{k,0}a,\ldots,C_{k,N-1}a)\in L^2(\omega)^N,\qquad \Gamma_k = \mathcal{O}_k^*\mathcal{O}_k = \sum_{m=0}^{N-1}C_{k,m}^*C_{k,m}.
\end{equation}
Thus
\begin{equation}\label{coefficientidentityN}
\sum_{m=0}^{N-1}\Vert C_{k,m}a\Vert_{L^2(\omega)}^2 
= \Vert\mathcal{O}_ka\Vert_{L^2(\omega)^N}^2 
= \langle\Gamma_ka,a\rangle_{\R^N}
= \big\Vert\Gamma_k^{1/2}a\big\Vert_{\R^N}^2.
\end{equation}
We call $\Gamma_k$ the matrix of the coefficient observation form. The uniform bounds on $\Psi_k$, $J_k$ and $B^*$ give $\sup_k\Vert\mathcal{O}_k\Vert<+\infty$ and $\sup_k\Vert\Gamma_k\Vert<+\infty$.

For $N=2$, with the observation-adapted bases of Lemma \ref{lem_Jordan}, one has
$$
C_{k,0}a = (a_1\theta_k+a_2\varphi_k)_{\vert\omega},\qquad C_{k,1}a = (I_ka_1\varphi_k)_{\vert\omega},\qquad \Gamma_k = \begin{pmatrix}\rho_k^2&0\\0&\alpha_k\end{pmatrix}.
$$
Hence \eqref{coefficientidentityN} is exactly \eqref{coefficientidentity2}. The scalar defect $\rho_k$ measures one diagonal direction of that form. 
For $N\geq3$, the form need not be diagonal in the flag-adapted basis, so the whole matrix must be considered in the sequel.

The block Fattorini--Hautus condition is
\begin{equation}\label{FHN}
\Gamma_k>0\qquad\forall k \geq 1,
\end{equation}
or equivalently the injectivity of every $\mathcal{O}_k$. For $N=2$ it reduces to \eqref{FH}, since $\alpha_k\geq m_\omega>0$. By the finite-dimensional Hautus criterion for the nilpotent matrix $J_k^*$, injectivity is equivalent to $\ker J_k^*\cap\ker(B^*\Psi_k)=\{0\}$. The Fattorini criterion therefore gives (see \cite{BoyerOlive_MCRF2014,Fattorini_SICON1966})
\begin{equation}\label{FHequivalenceN}
\Gamma_k>0\quad\forall k\geq 1 \quad\Leftrightarrow\quad \ker L_T^*=\{0\} \quad\Leftrightarrow\quad (A,B)\text{ is approximately controllable}.
\end{equation}
Here $L_T$ is defined by \eqref{globalgramian} for the present operators $A,B$. The moment estimate below also proves the equivalence with $\ker L_T^*=\{0\}$ by extracting the coefficients in \eqref{blockobservationN}.

\subsection{The spectral scale and the main theorem}
We use the same heat scale as in \eqref{Wsigma}, now with $N$ coefficients per block:
\begin{equation}\label{heatspaceN}
\Vert y\Vert_{W_\sigma}^2 = \sum_{k\geq1}e^{2\sigma k^2}\Vert c_k\Vert_{\R^N}^2,\qquad y = \sum_{k\geq1}\Phi_kc_k.
\end{equation}
The uniform bounds on $J_k$ imply that $e^{tJ_k}$ and $e^{-tJ_k}$ are bounded uniformly in $k$ at every fixed time. By \eqref{semigroupblocksN}, $e^{TA}:W_{-T}\to\mathcal{H}$ is an isomorphism, just as for two equations. The operator-range construction of Subsection \ref{sec_range_duality} therefore applies without change: $L_T$ and $G_T$ are given by \eqref{globalgramian}, $\mathcal{R}_T=\Ran L_T$ has norm \eqref{RTnorm}, and
$$
\mathcal{C}_T = \{y^0\in W_{-T}\ \mid\ -e^{TA}y^0\in\mathcal{R}_T\},\qquad \Vert y^0\Vert_{\mathcal{C}_T} = \Vert e^{TA}y^0\Vert_{\mathcal{R}_T}.
$$
In particular, the Gramian identity \eqref{gramianobservationidentity} holds. Under \eqref{FHN}, Lemma \ref{rangeHUM} applies by \eqref{FHequivalenceN} and gives the same criterion \eqref{rangecriterionCT} and dual norm formula \eqref{dualCT}.

Assume \eqref{FHN}. The inverse weight is determined by duality:
\begin{equation}\label{Gdualcost}
\sup_{a\neq0}\frac{\vert\langle c,a\rangle_{\R^N}\vert^2}{\Vert\Gamma_k^{1/2}a\Vert_{\R^N}^2} = \big\Vert\Gamma_k^{-1/2}c\big\Vert_{\R^N}^2 = \langle\Gamma_k^{-1}c,c\rangle_{\R^N}.
\end{equation}
For $\tau\in\R$, let $\mathcal{V}_\tau$ consist of all formal series with finite norm
\begin{equation}\label{VtauN}
\Vert y\Vert_{\mathcal{V}_\tau}^2 = \sum_{k\geq1}e^{-2\tau k^2}\big\Vert\Gamma_k^{-1/2}c_k\big\Vert_{\R^N}^2,\qquad y = \sum_{k\geq1}\Phi_kc_k.
\end{equation}
The scale is increasing in $\tau$ and embeds continuously and densely in $W_{-\tau}$, by the uniform upper bound on $\Gamma_k$ and the density of finite spectral sums. For $N=2$, its block weight is $\vert c_1\vert^2/\rho_k^2+\vert c_2\vert^2/\alpha_k$, uniformly equivalent to the simplified weight $\vert c_1\vert^2/\rho_k^2+\vert c_2\vert^2$ of Subsection \ref{scaletheorem2}. Thus the two definitions give the same spaces with equivalent norms.

\begin{theorem}\label{mainN}
Assume \eqref{FHN}. There exists $\kappa_N>0$ such that, for every $T>0$, there exist $c_T>0$ and $C_T>0$ such that
\begin{equation}\label{refinedboundsN}
c_T\sum_{k\geq1}k e^{-2Tk^2}\big\Vert\Gamma_k^{-1/2}c_k\big\Vert_{\R^N}^2 \leq \Vert y\Vert_{\mathcal{C}_T}^2 \leq C_T\sum_{k\geq1}e^{-2Tk^2+\kappa_Nk}\big\Vert\Gamma_k^{-1/2}c_k\big\Vert_{\R^N}^2
\end{equation}
for every finite spectral sum $y=\sum_k\Phi_kc_k$. The right-hand estimate extends to every series for which its right-hand side is finite, and the left-hand estimate holds for every $y\in\mathcal{C}_T$. Consequently, for every $\varepsilon>0$, one has the continuous inclusions
$$
\mathcal{V}_{T-\varepsilon}\hookrightarrow\mathcal{C}_T\hookrightarrow\mathcal{V}_T.
$$
\end{theorem}

This is Theorem \ref{main} with the scalar block weights replaced by the matrix weight $\Gamma_k^{-1}$. The heat factors and the subquadratic gap between $k$ and $e^{\kappa_Nk}$ are unchanged, the constant $\kappa_N$ being again the one of Proposition \ref{twosidedmoment}, which depends only on $N$. Before giving its proof in Subsection \ref{proofmainN}, we derive the modal critical times and compare $\mathcal{C}_T$ with the energy and heat scales in Subsections \ref{modaltimesN} and \ref{heatcomparisonN}, as in the two-equation case.

\begin{remark}[A sharper one-block estimate]\label{oneblockestimateN}
As in Remark \ref{oneblockestimate}, with a constant $c_T'>0$ independent of $k$, the factor $k$ in the necessary estimate improves to $k^2$ for data supported in a single block:
\begin{equation}\label{singleblockN}
\Vert\Phi_kc\Vert_{\mathcal{C}_T}^2 \geq c_T' k^2e^{-2Tk^2}\Vert\Gamma_k^{-1/2}c\Vert_{\R^N}^2
\end{equation}
for every $k\geq1$ and $c\in\R^N$. Indeed, applying \eqref{singlemoment} to \eqref{blockobservationN} gives the bound \eqref{singleN} below. Restricting \eqref{dualCT} to $z=\Psi_ka$ and optimizing over $a$, using \eqref{Gdualcost} and \eqref{semigroupblocksN}, gives \eqref{singleblockN} by the weighted translation estimate of Lemma \ref{translationN}, proved in Subsection \ref{proofmainN}. This estimate will be used to prove noninclusions at subcritical times.
\end{remark}

\subsection{Modal critical times and the energy space}\label{modaltimesN}
Assume \eqref{FHN} throughout this subsection.
For a closed subspace $X\subset\mathcal{H}$, its critical time $T_X$ is defined by \eqref{def_TX}. We now identify $N$ such times, extending the comparison with the energy space in Proposition \ref{Hcomparison}.

For two equations, \eqref{coefficientidentity2} gives $\Gamma_k^{1/2}=\mathrm{diag}(\rho_k,\sqrt{\alpha_k})$. Its eigenvalues are therefore $\rho_k$ and $\sqrt{\alpha_k}\asymp1$. After decreasing rearrangement, they satisfy $\mu_{1,k}\asymp1$ and $\mu_{2,k}\asymp\rho_k$, since $\rho_k$ is uniformly bounded. Formula \eqref{Tstarintro} gives the corresponding times $0$ and $T^\star$.

For $N$ equations, let
\begin{equation}\label{singularvaluesN}
\mu_{1,k}\geq\mu_{2,k}\geq\cdots\geq\mu_{N,k}>0
\end{equation}
be the eigenvalues of the positive matrix $\Gamma_k^{1/2}$. Equivalently, they are the singular values of $\mathcal{O}_k$, since $\Gamma_k=\mathcal{O}_k^*\mathcal{O}_k$ by \eqref{coefficientmatrixN}. Choose an orthonormal eigenbasis $(v_{1,k},\ldots,v_{N,k})$ such that $\Gamma_kv_{i,k}=\mu_{i,k}^2v_{i,k}$, and define
\begin{equation}\label{def_tau_i}
\tau_i=\limsup_{k\to +\infty}\frac{-\log\mu_{i,k}}{k^2}
\qquad\forall i\in\{1,\ldots,N\} .
\end{equation}
The uniform upper bound on $\Gamma_k$ and the directly observed vector in the $k$-th adjoint block show that
\begin{equation}\label{orderedtimesN}
0=\tau_1\leq\tau_2\leq\cdots\leq\tau_N\leq+\infty.
\end{equation}
Indeed, \eqref{flagnormalizationN} gives $\Psi_ke_N=\varphi_ke_N$, and $J_k^*e_N=0$ since $J_k$ is strictly upper triangular. Hence
$\langle\Gamma_ke_N,e_N\rangle_{\R^N} = \Vert C_{k,0}e_N\Vert_{L^2(\omega)}^2 = \alpha_k \geq m_\omega$.
Since $e_N$ is a unit vector, $m_\omega\leq\mu_{1,k}^2=\Vert\Gamma_k\Vert\leq C$. Thus $\log\mu_{1,k}$ is uniformly bounded, and \eqref{def_tau_i} yields $\tau_1=0$. The remaining inequalities in \eqref{orderedtimesN} follow from the ordering \eqref{singularvaluesN} and the uniform upper bound on all the singular values.
The number $\tau_i$ records the worst subsequential quadratic decay of $\mu_{i,k}$. Along any subsequence on which $\mu_{i,k}=e^{-\lambda k^2+\mathrm{o}(k^2)}$, the $i$-th modal direction is observed with that exponentially small amplitude and, by duality, the dual coefficient weight has size $e^{2\lambda k^2+\mathrm{o}(k^2)}$. At horizon $T$, the corresponding one-block control norm has Gaussian order $e^{(\lambda-T)k^2+\mathrm{o}(k^2)}$, by \eqref{refinedboundsN} and \eqref{singleblockN}. Thus \eqref{def_tau_i} selects the transition value.

For $1\leq i\leq N$, define the closed modal subspace
\begin{equation}\label{XiN}
X_i=\Big\{\sum_{k\geq1}d_k\Phi_kv_{i,k}\ \mid\ (d_k)\in\ell^2\Big\}.
\end{equation}
Thus $X_i$ selects in every block the $i$-th singular direction, from the best observed direction $i=1$ to the worst one $i=N$. The blockwise orthogonal change of coordinates preserves the Riesz property, so
$\mathcal{H}=X_1\oplus\cdots\oplus X_N$. On $X_i$, \eqref{VtauN} becomes
\begin{equation}\label{XinormN}
\Big\Vert\sum_kd_k\Phi_kv_{i,k}\Big\Vert_{\mathcal{V}_\tau}^2
=\sum_{k\geq1}e^{-2\tau k^2}\mu_{i,k}^{-2}\vert d_k\vert^2.
\end{equation}
The modal spaces depend on the chosen eigenvectors. The statements below hold for every such choice; the effect of changing block coordinates is discussed at the end of this subsection.

\paragraph{Critical-time interpretation.}
The next theorem turns the observation loss into a precise controllability statement. Crossing $\tau_i$ makes the entire Hilbert subspace $X_i$ uniformly null controllable; below $\tau_i$, even inclusion of $X_i$ in the controllable-data space fails.

\begin{theorem}[Modal critical times]\label{modalcriticalN}
Assume \eqref{FHN}. For the modal spaces $X_i$ defined by \eqref{XiN}, one has, for every $1\leq i\leq N$,
\begin{align*}
T>\tau_i&\ \Rightarrow\ X_i\hookrightarrow\mathcal{C}_T,\\
T<\tau_i&\ \Rightarrow\ X_i\not\subset\mathcal{C}_T.
\end{align*}
\end{theorem}
No general statement is made when $T=\tau_i$.

\begin{proof}
If $T>\tau_i$, choose $\tau$ with $\tau_i<\tau<T$. Definition \eqref{def_tau_i} and identity \eqref{XinormN} give $X_i\hookrightarrow\mathcal{V}_\tau$, and Theorem \ref{mainN} gives $X_i\hookrightarrow\mathcal{C}_T$.

Conversely, suppose that $T<\tau_i$ and $X_i\subset\mathcal{C}_T$. Both spaces embed continuously in $W_{-T}$, so the inclusion is continuous by the closed graph theorem. Applied to $\Phi_kv_{i,k}$, Remark \ref{oneblockestimateN} gives, by \eqref{singleblockN},
$1\geq ck^2e^{-2Tk^2}\mu_{i,k}^{-2}$ after absorbing the inclusion norm into $c$. Hence $\tau_i\leq T$, a contradiction.
\end{proof}

\begin{remark}[Comparison with the energy space]
The conclusions of Proposition \ref{Hcomparison} for $\mathcal{C}_T$ extend with $\tau_N$ in place of $T^\star$. Indeed, since $\mu_{1,k}\asymp1$, the datum with coefficients $c_k=e^{Tk^2-\kappa_Nk}v_{1,k}/k$ satisfies the sufficient condition in \eqref{refinedboundsN} but does not belong to $\mathcal{H}$. Thus $\mathcal{C}_T$ is never contained in $\mathcal{H}$. If $T>\tau_N$, all $X_i$ embed continuously in $\mathcal{C}_T$, so $\mathcal{H}\hookrightarrow\mathcal{C}_T$ continuously and properly. If $T<\tau_N$, Theorem \ref{modalcriticalN} gives $X_N\not\subset\mathcal{C}_T$, and $\mathcal{H}$ and $\mathcal{C}_T$ are incomparable. The case $T=\tau_N$ remains undecided.
\end{remark}

\begin{remark}[Nested modal subspaces]
For $1\leq m\leq N$, set $Y_m=X_1\oplus\cdots\oplus X_m$. Then
$Y_1\subset\cdots\subset Y_N=\mathcal{H}$ and $T_{Y_m}=\tau_m$.
Indeed, the critical time of a finite direct sum is the largest critical time of its summands. In particular, $X_1\hookrightarrow\mathcal{C}_T$ for every $T>0$. If $1\leq m<N$ and $\tau_m<T<\tau_{m+1}$, then $Y_m\hookrightarrow\mathcal{C}_T$ whereas $Y_{m+1}\not\subset\mathcal{C}_T$: passing the modal thresholds produces successive gains of uniform controllability for whole nested subspaces. Coincident values correspond to simultaneous gains; the cases $T=\tau_m$ remain undecided.
\end{remark}

\begin{remark}[Choice of block coordinates]
The times $\tau_i$ and the coefficient cost are ``intrinsic'' in the sense that they are invariant under uniformly conditioned changes of dual block coordinates. Indeed, if $a=P_k\widetilde a$, then $\widetilde\Psi_k=\Psi_kP_k$ and $\widetilde\Gamma_k=P_k^*\Gamma_kP_k$. The dual primal map is $\widetilde\Phi_k=\Phi_kP_k^{-*}$, so the same primal vector has coordinates $\widetilde c=P_k^*c$ and $\Vert\widetilde\Gamma_k^{-1/2}\widetilde c\Vert_{\R^N}^2 = \Vert\Gamma_k^{-1/2}c\Vert_{\R^N}^2$.
If $P_k$ and $P_k^{-1}$ are uniformly bounded, the heat norms remain equivalent. The min-max principle makes the ordered eigenvalues of $\widetilde\Gamma_k^{1/2}$ and $\Gamma_k^{1/2}$ comparable up to constants independent of $k$, so \eqref{def_tau_i} gives the same times.

The individual spaces $X_i$ depend on the normalized block coordinates and on the chosen eigenvectors. For a multiple eigenvalue, only the sum of the corresponding modal directions is independent of rotations within its eigenspace. Moreover, $X_i$ is generally not invariant under the semigroup, since $v_{i,k}$ need not span a $J_k$-invariant line. Theorem \ref{modalcriticalN} holds for every admissible choice.
\end{remark}

\subsection{Directional and isotropic heat scales}\label{heatcomparisonN}
The analogue of Corollary \ref{Wphase} now distinguishes the singular directions. For $\sigma\in\R$ and $1\leq i\leq N$, define
\begin{equation}\label{directionalheatN}
\mathcal{W}_\sigma^{(i)}
=\Big\{\sum_{k\geq1}d_k\Phi_kv_{i,k}\ \mid\  \sum_{k\geq1}e^{2\sigma k^2}\vert d_k\vert^2<+\infty\Big\}.
\end{equation}
In \eqref{directionalheatN}, $\mathcal{W}_0^{(i)}=X_i$, while $\sigma>0$ imposes heat regularity along the $i$-th singular direction.

\begin{corollary}[Directional and isotropic heat scales]
Assume \eqref{FHN}, and let $T>0$ and $\sigma\geq-T$.
For every $1\leq i\leq N$,
\begin{align}
T+\sigma>\tau_i&\ \Rightarrow\ \mathcal{W}_\sigma^{(i)}\hookrightarrow\mathcal{C}_T,\label{directionalpositiveN}\\
T+\sigma<\tau_i&\ \Rightarrow\ \mathcal{W}_\sigma^{(i)}\not\subset\mathcal{C}_T.\label{directionalnegativeN}
\end{align}
Moreover,
\begin{equation}\label{directionaldecompositionN}
W_\sigma=\mathcal{W}_\sigma^{(1)}\oplus\cdots\oplus\mathcal{W}_\sigma^{(N)},
\end{equation}
and therefore 
\begin{align}
T+\sigma>\tau_N&\ \Rightarrow\ W_\sigma\hookrightarrow\mathcal{C}_T,\label{isotropicpositiveN}\\
T+\sigma<\tau_N&\ \Rightarrow\ W_\sigma\not\subset\mathcal{C}_T.\label{isotropicnegativeN}
\end{align}
In particular, $T_{\mathcal{H}}=\tau_N$. No general conclusion is asserted when $T+\sigma=\tau_i$.
\end{corollary}

\begin{proof}
If $T+\sigma>\tau_i$, choose $\tau<T$ with $\tau+\sigma>\tau_i$. Equations \eqref{def_tau_i} and \eqref{XinormN} give $\mathcal{W}_\sigma^{(i)}\hookrightarrow\mathcal{V}_\tau$, and Theorem \ref{mainN} concludes. Conversely, if $\mathcal{W}_\sigma^{(i)}\subset\mathcal{C}_T$, the closed graph theorem makes the inclusion continuous. Applying \eqref{singleblockN} to $\Phi_kv_{i,k}$ gives
$e^{2\sigma k^2}\geq ck^2e^{-2Tk^2}\mu_{i,k}^{-2}$, up to a fixed multiplicative constant, and hence $\tau_i\leq T+\sigma$.
The blockwise orthogonal decomposition in the norm \eqref{heatspaceN} gives \eqref{directionaldecompositionN}; no single directional space equals $W_\sigma$ when $N>1$. Since $\tau_N$ is the largest of the modal times, \eqref{directionalpositiveN} for all $i$ proves \eqref{isotropicpositiveN}, while \eqref{directionalnegativeN} with $i=N$ proves \eqref{isotropicnegativeN}. Taking $\sigma=0$ identifies the infimum $T_{\mathcal{H}}$.
\end{proof}

For $N=2$, we recover Corollary \ref{Wphase}, since $\tau_2=T^\star$. For general $N$, the new information is the individual threshold $\tau_i$ on each directional scale.

\subsection{Proof of Theorem \ref{mainN}}\label{proofmainN}
We follow the two steps of Subsection \ref{proofmain}, with $\R^N$ in place of $\R^2$.

\paragraph{The adjoint observation in block coordinates.}
For a finitely supported sequence $a=(a_k)_{k\geq1}$ in $\R^N$, set $z[a]=\sum_{k\geq1}\Psi_ka_k$. Formula \eqref{blockobservationN} gives
\begin{equation}\label{globalobservationN}
B^*e^{sA^*}z[a] = \sum_{k\geq1}e^{-k^2s}\sum_{m=0}^{N-1}s^m A_{m,k},\qquad A_{m,k}=C_{k,m}a_k,
\end{equation}
for any $s\in[0,T]$.
This is \eqref{globalobservation2} with $N$ polynomial coefficients instead of two. By \eqref{coefficientidentityN},
$$
\sum_{m=0}^{N-1}\Vert A_{m,k}\Vert_{L^2(\omega)}^2 = \big\Vert\Gamma_k^{1/2}a_k\big\Vert_{\R^N}^2.
$$
We now repeat the passage from \eqref{coefficientidentity2} to \eqref{gramiancomparison2}.

\begin{lemma}[Frequency-diagonal estimates for the global Gramian]\label{blockN}
Let $\kappa_N$ be the constant in Proposition \ref{twosidedmoment}. For every $T>0$, there are positive constants $c_{T,N}$ and $C_{T,N}$ such that
\begin{multline}\label{doubleN}
c_{T,N}\sum_{k\geq1}e^{-\kappa_Nk}\big\Vert\Gamma_k^{1/2}a_k\big\Vert_{\R^N}^2 \leq \langle G_Tz[a],z[a]\rangle_{\mathcal{H}}\\
= \int_0^T\big\Vert B^*e^{sA^*}z[a]\big\Vert_{L^2(\omega)}^2\,ds \leq C_{T,N}\sum_{k\geq1}\frac{1}{k}\big\Vert\Gamma_k^{1/2}a_k\big\Vert_{\R^N}^2
\end{multline}
for every finitely supported sequence $a$. For a single block,
\begin{equation}\label{singleN}
\langle G_T\Psi_ka,\Psi_ka\rangle_{\mathcal{H}} \leq \frac{C}{k^2}\big\Vert\Gamma_k^{1/2}a\big\Vert_{\R^N}^2,
\end{equation}
where $C>0$ is independent of $T$ and $k$.
\end{lemma}

\begin{proof}
Substitute \eqref{globalobservationN} into \eqref{gramianobservationidentity} and apply Proposition \ref{twosidedmoment} with $X=L^2(\omega)$. Identity \eqref{coefficientidentityN} turns \eqref{twosidedmomentineq} into \eqref{doubleN}, and \eqref{singlemoment} gives \eqref{singleN}. These estimates do not require \eqref{FHN}.
\end{proof}

As in the observation step of Subsection \ref{proofmain}, completeness of the adjoint Riesz basis and the uniform upper bound on $\Gamma_k$ allow spectral truncation to extend \eqref{doubleN} in Lemma \ref{blockN} to every $a\in\ell^2(\R^N)$. The sums are diagonal in $k$: the moment estimate handles interactions between frequencies, while $\Gamma_k$ retains the form on the coefficients within each block.

\begin{remark}[Comparison with the diagonal blocks of the Gramian]
Let $S:\ell^2(\R^N)\to\mathcal{H}$ be the map $Sa=\sum_{k\geq1}\Psi_ka_k$. The matrix $S^*G_TS$ represents the quadratic form of $G_T$ in the adjoint Riesz coordinates. Its diagonal blocks are $G_{T,k} = \Psi_k^*G_T\Psi_k$, so $G_{T,k}$ represents the restriction of this form to the generalized eigenspace $E_k^*$ of $A^*$. This interpretation uses a congruence, since the Riesz basis need not be orthonormal.

Fix $T>0$. Formula \eqref{blockobservationN} and the change of variable $r=k^2s$ give, for every $b\in\R^N$,
\begin{multline*}
\langle G_{T,k}b,b\rangle_{\R^N} = \int_0^T \Vert B^*e^{sA^*}\Psi_kb\Vert_{L^2(\omega)}^2\,ds
= \frac{1}{k^2}\int_0^{k^2T}e^{-2r}\Big\Vert\sum_{m=0}^{N-1}r^mk^{-2m}C_{k,m}b\Big\Vert_{L^2(\omega)}^2\,dr \\
\asymp_{T,N} \frac{1}{k^2}\sum_{m=0}^{N-1}k^{-4m}\Vert C_{k,m}b\Vert_{L^2(\omega)}^2.
\end{multline*}
The equivalence is uniform in $k$ and $b$: the integral over $[0,k^2T]$ is bounded below by that over $[0,T]$ and above by that over $[0,+\infty)$, and the corresponding scalar polynomial moment matrices are positive definite. Together with \eqref{coefficientidentityN}, this yields
$$
c_{T,N}k^{-(4N-2)}\Gamma_k \leq G_{T,k} \leq C_Nk^{-2}\Gamma_k
$$
in the sense of quadratic forms.
Consequently, \eqref{doubleN} implies that, for every finitely supported sequence $a$,
$$
c_{T,N}'\sum_{k\geq1}k^2e^{-\kappa_Nk}\langle G_{T,k}a_k,a_k\rangle_{\R^N}
\leq \langle G_Tz[a],z[a]\rangle_{\mathcal{H}}
\leq C_{T,N}'\sum_{k\geq1}k^{4N-3}\langle G_{T,k}a_k,a_k\rangle_{\R^N}.
$$
Thus the moment estimates compare the full Gramian with weighted sums of its diagonal blocks. They control the collective effect of the off-diagonal blocks through the displayed weights; no uniform equivalence with the unweighted block-diagonal form is asserted. In the proof, the matrices $\Gamma_k$ provide an explicit description in terms of the observed polynomial coefficients, independent of $T$ and suited to the duality argument.
\end{remark}

\paragraph{The numerator and duality argument.}
For $y^0\in W_{-T}$, write $y^0=\sum_{k\geq1}\Phi_kc_k$ and set $M_k(T)=e^{TJ_k}$. Formula \eqref{semigroupblocksN} and biorthogonality give
\begin{equation}\label{numeratorN}
\langle e^{TA}y^0,z[a]\rangle_{\mathcal{H}} = \sum_{k\geq1}e^{-Tk^2}\langle M_k(T)c_k,a_k\rangle_{\R^N}.
\end{equation}
This is the same numerator as \eqref{numerator2}. The next lemma generalizes the weighted triangular estimate \eqref{horizonequiv}; it is the only additional step needed before repeating the duality proof.

\begin{lemma}[Translation of the block polynomial]\label{translationN}
Assume \eqref{FHN}. For every fixed $T\in\R$, there exists $\kappa_T\geq1$ such that
\begin{equation}\label{translationequiv}
\kappa_T^{-1}\Vert\Gamma_k^{-1/2}c\Vert_{\R^N}^2 \leq \Vert\Gamma_k^{-1/2}M_k(T)c\Vert_{\R^N}^2 \leq \kappa_T\Vert\Gamma_k^{-1/2}c\Vert_{\R^N}^2
\end{equation}
for every $k\geq1$ and $c\in\R^N$.
\end{lemma}

\begin{proof}
Set $P_{k,a}(s)=\sum_{m=0}^{N-1}s^m C_{k,m}a$. The definition of the coefficient maps gives
$$
P_{k,M_k(T)^*a}(s) = P_{k,a}(s+T).
$$
Here $M_k(T)^*$ is the transpose of $M_k(T)$, consistently with the adjoint pairing in \eqref{numeratorN}.

Translation by $T$ transforms the coefficient of $s^j$ into $\sum_{m=j}^{N-1}\binom{m}{j}T^{m-j}C_{k,m}a$. This triangular map on $L^2(\omega)^N$ and its inverse, translation by $-T$, have norms depending only on $T$ and $N$. Therefore \eqref{coefficientidentityN} gives
$$
\Vert\Gamma_k^{1/2}M_k(T)^*a\Vert_{\R^N} \asymp_{T,N} \Vert\Gamma_k^{1/2}a\Vert_{\R^N}.
$$
Using \eqref{Gdualcost} and changing variables $b=M_k(T)^*a$, we obtain
$$
\Vert\Gamma_k^{-1/2}M_k(T)c\Vert_{\R^N} = \sup_{a\neq0}\frac{\vert\langle c,M_k(T)^*a\rangle_{\R^N}\vert}{\Vert\Gamma_k^{1/2}a\Vert_{\R^N}} \asymp_{T,N} \Vert\Gamma_k^{-1/2}c\Vert_{\R^N}.
$$
Squaring proves \eqref{translationequiv}. For $N=2$, translation changes $(C_{k,0}a,C_{k,1}a)$ into $(C_{k,0}a+TC_{k,1}a,C_{k,1}a)$; this is the coefficient version of the triangular correction used in \eqref{horizonequiv}.
\end{proof}

\begin{proof}[Proof of Theorem \ref{mainN}]
Assume first that the series on the right-hand side of \eqref{refinedboundsN} is finite. Since $\Gamma_k$ is uniformly bounded, the corresponding datum belongs to $W_{-T}$. Write each summand of \eqref{numeratorN} as
$$
e^{-Tk^2}\langle M_k(T)c_k,a_k\rangle_{\R^N} = \big\langle e^{-Tk^2+\kappa_Nk/2}\Gamma_k^{-1/2}M_k(T)c_k,e^{-\kappa_Nk/2}\Gamma_k^{1/2}a_k\big\rangle_{\R^N}.
$$
The Cauchy--Schwarz inequality in $\ell^2(\R^N)$ gives
$$
\vert\langle e^{TA}y^0,z[a]\rangle_{\mathcal{H}}\vert^2 \leq \Big(\sum_{k\geq1}e^{-2Tk^2+\kappa_Nk}\Vert\Gamma_k^{-1/2}M_k(T)c_k\Vert_{\R^N}^2\Big)\Big(\sum_{k\geq1}e^{-\kappa_Nk}\big\Vert\Gamma_k^{1/2}a_k\big\Vert_{\R^N}^2\Big).
$$
The first factor is finite by \eqref{translationequiv}. The lower bound in \eqref{doubleN} controls the second by $c_{T,N}^{-1}$$\langle G_Tz[a],z[a]\rangle_{\mathcal{H}}$. By density of finite spectral test vectors and continuity of the numerator and of $L_T^*$, the quotient in \eqref{rangecriterionCT} is bounded on all of $\mathcal{H}$. Lemma \ref{rangeHUM} yields $y^0\in\mathcal{C}_T$ and
$$
\Vert y^0\Vert_{\mathcal{C}_T}^2 \leq c_{T,N}^{-1}\sum_{k\geq1}e^{-2Tk^2+\kappa_Nk}\Vert\Gamma_k^{-1/2}M_k(T)c_k\Vert_{\R^N}^2.
$$
Using \eqref{translationequiv} proves the right-hand inequality in \eqref{refinedboundsN}.

For the left-hand inequality, let $y^0\in\mathcal{C}_T$ and let $F\subset\N\setminus\{0\}$ be finite. Choose $a_k=k e^{-Tk^2}\Gamma_k^{-1}M_k(T)c_k$ for $k\in F$ and $a_k=0$ otherwise, and set $z_F=z[a]$. Exactly as in the two-equation proof, \eqref{numeratorN} gives
$$
\langle e^{TA}y^0,z_F\rangle_{\mathcal{H}} = \sum_{k\in F}k e^{-2Tk^2}\Vert\Gamma_k^{-1/2}M_k(T)c_k\Vert_{\R^N}^2 = S_F,
$$
whereas the upper bound in \eqref{doubleN} gives $\langle G_Tz_F,z_F\rangle_{\mathcal{H}}\leq C_{T,N}S_F$. Formula \eqref{dualCT} implies $S_F^2\leq C_{T,N}\Vert y^0\Vert_{\mathcal{C}_T}^2S_F$. Dividing by $S_F$ if it is positive, and using \eqref{translationequiv}, yields
$$
\Vert y^0\Vert_{\mathcal{C}_T}^2 \geq C_{T,N}^{-1}S_F \geq (C_{T,N}\kappa_T)^{-1}\sum_{k\in F}k e^{-2Tk^2}\Vert\Gamma_k^{-1/2}c_k\Vert_{\R^N}^2.
$$
The same inequality holds when $S_F=0$. Take $F=\{1,\ldots,n\}$ and let $n\to+\infty$ to obtain the left-hand inequality in \eqref{refinedboundsN}. Discarding $k\geq1$ gives $\mathcal{C}_T\hookrightarrow\mathcal{V}_T$, and $e^{\kappa_Nk}\leq C_{\varepsilon}e^{2\varepsilon k^2}$ in the sufficient bound gives $\mathcal{V}_{T-\varepsilon}\hookrightarrow\mathcal{C}_T$.
\end{proof}

\subsection{The three-equation block: consecutive and direct couplings}

We illustrate the preceding spectral description and modal critical times for three equations:
\begin{equation}\label{system3}
\left\{
\begin{array}{ll}
\partial_t y_1 = \partial_{xx}y_1+\beta_{12}(x)y_2+\beta_{13}(x)y_3,\\
\partial_t y_2 = \partial_{xx}y_2+\beta_{23}(x)y_3,\\
\partial_t y_3 = \partial_{xx}y_3+1_\omega u,
\end{array}
\right.
\end{equation}
with homogeneous Dirichlet conditions. In \eqref{system3}, a term $\beta_{ij}y_j$ in the $i$-th equation is a directed coupling from component $j$ to component $i$. Thus $3\to2\to1$ is the consecutive path from the controlled component to the first one, while $\beta_{13}y_3$ adds the direct path $3\to1$.

The flag-adapted adjoint basis $(\Psi_{1,k},\Psi_{2,k},\Psi_{3,k})$ of Lemma \ref{rieszN} satisfies
\begin{equation}\label{J3coordinates}
\begin{aligned}
(A^*+k^2)\Psi_{1,k}&=\eta_{12,k}\Psi_{2,k}+\eta_{13,k}\Psi_{3,k},\\
(A^*+k^2)\Psi_{2,k}&=\eta_{23,k}\Psi_{3,k},\\
(A^*+k^2)\Psi_{3,k}&=0.
\end{aligned}
\end{equation}
Relations \eqref{J3coordinates} generalize the second line of \eqref{Jordanrelations}; the adapted bases are chosen real, so that the coefficients $\eta_{ij,k}$ are the entries $(J_k)_{ij}$ of the real matrix $J_k$. The coefficients $\eta_{12,k}$ and $\eta_{23,k}$ are the consecutive coupling moments: by the superdiagonal formula, $\eta_{12,k}=\int_0^\pi\beta_{12}\varphi_k^2\,dx$ and $\eta_{23,k}=\int_0^\pi\beta_{23}\varphi_k^2\,dx$. In the chosen flag-adapted coordinates, the coefficient $\eta_{13,k}$ contains the direct contribution of $\beta_{13}$ and the corrections produced by the spatial generalized eigenvectors; it may therefore be nonzero even for a spatial nearest-neighbor chain.

Set $h_{j,k}=B^*\Psi_{j,k}\in L^2(\omega)$. The flag normalization gives $h_{3,k}=\varphi_{k\vert\omega}$. For $a=(a_1,a_2,a_3)\in\R^3$, the coefficient operators defined above are
\begin{align*}
C_{k,0}a&=a_1h_{1,k}+a_2h_{2,k}+a_3h_{3,k},\\
C_{k,1}a&=\eta_{12,k}a_1h_{2,k}
+(\eta_{13,k}a_1+\eta_{23,k}a_2)h_{3,k},\\
C_{k,2}a&=\frac{\eta_{12,k}\eta_{23,k}}{2}a_1h_{3,k}.
\end{align*}
Accordingly,
$$
B^*e^{sA^*}\sum_{j=1}^3a_j\Psi_{j,k}
=e^{-k^2s}\big(C_{k,0}a+sC_{k,1}a+s^2C_{k,2}a\big),
$$
and the matrix $\Gamma_k$ represents the coefficient form
$$
\Vert\Gamma_k^{1/2}a\Vert_{\R^3}^2 = \Vert C_{k,0}a\Vert_{L^2(\omega)}^2 + \Vert C_{k,1}a\Vert_{L^2(\omega)}^2 + \Vert C_{k,2}a\Vert_{L^2(\omega)}^2.
$$
Compared with \eqref{coefficientidentity2}, a quadratic coefficient has been added. Mixed terms in the components of $a$ can occur inside each $\Vert C_{k,m}a\Vert_{L^2(\omega)}^2$: the functions $h_{j,k}=B^*\Psi_{j,k}$ appearing there need not be orthogonal, and several components of $a$ may multiply the same function, as in $C_{k,1}a$. Hence $\Gamma_k$ is generally not diagonal. The form contains no inner products between coefficients of different polynomial degrees.

To illustrate this, consider first the idealized configuration $h_{1,k}=h_{2,k}=0$ and $\eta_{13,k}=0$. After identifying $\Span(h_{3,k})$ with $\R$, the eigenvalues of $\Gamma_k^{1/2}$, up to ordering, are
$$
\Vert h_{3,k}\Vert_{L^2(\omega)},\qquad
\vert\eta_{23,k}\vert\Vert h_{3,k}\Vert_{L^2(\omega)},\qquad
\frac{\vert\eta_{12,k}\eta_{23,k}\vert}{2}
\Vert h_{3,k}\Vert_{L^2(\omega)}.
$$
If
$$
-\log\vert\eta_{12,k}\vert=\tau k^2+\mathrm{o}(k^2),
\qquad
-\log\vert\eta_{23,k}\vert=\tau k^2+\mathrm{o}(k^2),
$$
the three modal critical times defined in \eqref{def_tau_i} are $0$, $\tau$ and $2\tau$. In this special chain, the time associated with a component is the sum of the exponential losses along the consecutive path that separates it from the controlled component.

Now allow the direct coupling $3\to1$. Even with $h_{1,k}=h_{2,k}=0$, the last two scalar coefficients contain the matrix
$$
\begin{pmatrix}
\eta_{13,k}&\eta_{23,k}\\
\eta_{12,k}\eta_{23,k}/2&0
\end{pmatrix}.
$$
Its determinant is $-\eta_{12,k}\eta_{23,k}^2/2$, but its largest singular value depends on $\eta_{13,k}$. Since the product of its singular values is the absolute value of the determinant, a large direct contribution can improve one combination while another becomes nearly cancelling and harder to observe. The products of consecutive coupling moments alone do not determine the hierarchy; the eigenvalues of $\Gamma_k^{1/2}$ do.

\subsection{Exact inverse design by nonlocal couplings}
The following proposition realizes any prescribed finite hierarchy. The price for this exact inverse design is that the couplings are bounded nonlocal operators rather than multiplication operators.

\begin{remark}\label{abstractscopeN}
The proof of Theorem \ref{mainN} uses a dual Riesz decomposition with uniformly bounded nilpotent block matrices $J_k$ and the polynomial observation formula \eqref{blockobservationN}. Together with \eqref{FHN}, these properties suffice for the range, moment and duality arguments. The conclusions therefore apply also to systems whose couplings are bounded operators rather than multiplication operators, whenever these block hypotheses hold. In fact, the proof of Lemma \ref{rieszN} given in Subsection \ref{app_proof_lem_rieszN} only uses that the coupling $Q = A-\partial_{xx}\Id_{\R^N}$ is a bounded strictly upper triangular operator on $\mathcal{H}$; it thus provides these block hypotheses for every such coupling, the superdiagonal formula becoming $(J_k)_{j,j+1} = \langle Q_{j,j+1}\varphi_k,\varphi_k\rangle_{L^2(0,\pi)}$. This applies to Proposition \ref{inverse-designN}.
\end{remark}

\begin{proposition}\label{inverse-designN}
Let $\Delta_D=\partial_{xx}$ be the one-dimensional Dirichlet Laplacian on $(0,\pi)$, and let $\gamma_1,\ldots,\gamma_{N-1}\geq0$. Consider
\begin{align*}
\partial_ty_j&=\Delta_Dy_j+e^{\gamma_j\Delta_D}y_{j+1}, \qquad 1\leq j<N, \\
\partial_ty_N&=\Delta_Dy_N+1_\omega u.
\end{align*}
This system satisfies the block Fattorini--Hautus condition, and its ordered modal critical times are
\begin{equation}\label{designedtimesN}
\tau_i=\sum_{j=N+1-i}^{N-1}\gamma_j,
\qquad 1\leq i\leq N,
\end{equation}
where an empty sum is zero. Consequently, for every prescribed sequence
$0=\vartheta_1\leq\vartheta_2\leq\cdots\leq\vartheta_N<+\infty$,
the choice
\begin{equation}\label{inversegammaN}
\gamma_j=\vartheta_{N-j+1}-\vartheta_{N-j},
\qquad 1\leq j<N,
\end{equation}
gives $\tau_i=\vartheta_i$ for all $i$. 
\end{proposition}

\begin{proof}
The operators $e^{\gamma_j\Delta_D}$ are bounded and commute with $\Delta_D$. Hence the standard Fourier basis gives the generalized eigenspaces directly, without spatial correctors, and
\begin{equation}\label{JnonlocalN}
(J_k)_{j,j+1}=e^{-\gamma_jk^2},
\qquad (J_k)_{i,j}=0\quad\text{if }j\neq i+1.
\end{equation}
For the coefficient maps defined above, one obtains, for $1\leq q\leq N$,
\begin{equation}\label{CnonlocalN}
C_{k,m}e_q=0\quad\text{if }m\neq N-q,
\qquad
C_{k,N-q}e_q
=\frac{e^{-k^2\sum_{j=q}^{N-1}\gamma_j}}{(N-q)!}
\varphi_{k\vert\omega}.
\end{equation}
The vectors $\mathcal{O}_ke_q$ occupy distinct polynomial-coefficient slots and are therefore orthogonal, so $\Gamma_k$ is diagonal. Since $\Vert\varphi_k\Vert_{L^2(\omega)}$ is bounded above and below, all their norms are positive and their quadratic exponential orders are
$0$, $\gamma_{N-1}$, $\gamma_{N-2}+\gamma_{N-1}$, $\ldots$, $\sum_{j=1}^{N-1}\gamma_j$.
Ordering the eigenvalues of $\Gamma_k^{1/2}$ from largest to smallest proves \eqref{designedtimesN}. Formula \eqref{inversegammaN} then gives a telescoping sum. This system satisfies the abstract block hypotheses stated in Remark \ref{abstractscopeN}: its Fourier basis is orthonormal, the matrices \eqref{JnonlocalN} are uniformly bounded and nilpotent, and the coefficient maps in \eqref{CnonlocalN} are uniformly bounded. The block moment, range-duality and one-block arguments therefore apply, and Theorem \ref{modalcriticalN} identifies the numbers in \eqref{designedtimesN} with the critical times of the corresponding modal subspaces.
\end{proof}

For example, when $N=3$ and $\gamma_1=\gamma_2=\tau>0$, the hierarchy is $(\tau_1,\tau_2,\tau_3)=(0,\tau,2\tau)$.

\section{Conclusion and open questions}

At each fixed horizon, the almost sharp spectral description of $\mathcal{C}_T$ combines heat dissipation with the cost of observing each generalized eigenspace. For two equations this gives the $L^2$ threshold $T^\star$ and, below it, anisotropic spaces that may be incomparable with standard regularity scales. For finite cascades, the singular values of the coefficient observation maps yield the critical times of adapted modal subspaces. The full energy-space threshold is the worst one, $\tau_N$, while $\tau_1,\ldots,\tau_{N-1}$ describe successive gains of controllability on nested subspaces. The exact inverse-design model shows that all these times can be distinct. For local couplings, consecutive moments do not determine the hierarchy; the full coefficient map, including spatial correctors and cancellations, must be retained.

\paragraph{At the critical times.}
Theorem \ref{main} leaves a subquadratic gap between the sufficient and necessary weights. The linear exponential $e^{\kappa_2k}$ of the sufficient weight is inherited from the scalar biorthogonal family and cannot be removed at that level: already for the family $(e^{-k^2t})_{k\geq1}$ in $L^2(0,+\infty)$, the squared norm of the $k$-th biorthogonal element is the classical product
$$
2k^2\prod_{j\neq k}\Big(\frac{j^2+k^2}{j^2-k^2}\Big)^2
$$
(diagonal entry of the inverse of the Cauchy Gram matrix $(1/(j^2+\ell^2))_{j,\ell\geq1}$), which is equal to $\frac{2}{\pi^2}\sinh^2(\pi k)$ by the product formulae for $\sinh$ and $\sin$: the biorthogonal family grows exactly like $e^{\pi k}$, and restricting the time interval to $(0,T)$ only increases these norms. Closing the gap requires sharper two-sided estimates for the vector-valued moment problem associated with the observed functions, taking into account the spatial functions $\theta_{k\vert\omega}$ and $\varphi_{k\vert\omega}$ and not only the time exponentials. In particular, it is not known whether the factors $e^{Ck}$ and $k$ are technical losses or occur in the exact range norm. The same issue governs the inclusions at $T=T^\star$ and $T+\sigma=T^\star$. Already for the arithmetic couplings of Subsection \ref{sec_3examples} with $0<\vartheta<+\infty$, we do not know whether $\mathcal{H}\subset\mathcal{C}_{\vartheta}$. A separate question is the behavior of the uniform null-control cost on $\mathcal{H}$ as $T\downarrow T^\star$; even for a scalar profile-controlled heat equation, the control cost may blow up at an arbitrarily prescribed rate \cite{Lissy_JMAA2017}.

\paragraph{Several directed coupling paths.}
For $N\geq3$, even a nearest-neighbor chain produces spatial correctors in its generalized eigenvectors. Additional upper triangular entries create distinct directed paths in the coupling graph from the controlled component to a given uncontrolled component. Their contributions can enter the same observed polynomial coefficient and cancel. The immediate problem is to find checkable assumptions on the supports and signs of the $\beta_{ij}$ that yield two-sided estimates for the singular values of the full block map. Such estimates would decide whether local multiplication couplings can realize several distinct positive critical times, or even a prescribed hierarchy. A separate problem is to compare the resulting adapted modal spaces with the physical component spaces.

\paragraph{Choice of the coupling.}
The space $\mathcal{C}_T$ depends on $\beta$ and $\omega$ only through the sequence $(\rho_k)$, and, for $N$ equations, on the couplings $\beta_{ij}$ only through the matrices $\Gamma_k$. This dependence is not robust: in the arithmetic example of Subsection \ref{sec_3examples}, indicator functions supported in two adjacent intervals, arbitrarily close to each other in $L^1(0,\pi)$, may have critical times $0$ and $+\infty$. Two questions follow. Given a class of admissible couplings, for instance the indicator functions of the subsets of $(0,\pi)\setminus\omega$ of prescribed measure, which couplings make $\mathcal{C}_T$ largest, or make the null-control cost on a given space smallest, and for which couplings does \eqref{FH} fail? This is, for the coupling, the analogue of the optimal observation problems of \cite{PrivatTrelatZuazua_ARMA2015}, in which the observability constant is optimized over the position of the sensor; here the unknown is the coupling, and the objective is a whole sequence of modal coefficients. Time-dependent couplings $\beta(t,x)$, in the spirit of the switching controls of \cite{Zuazua_JEMS2011}, may remove the nondegeneracy condition \eqref{FH} altogether, by exciting every mode on some time interval even when a fixed profile cancels it; the block structure is then lost, and the analysis has to rely on the range duality of Subsection \ref{sec_range_duality} directly.

\paragraph{Beyond triangular couplings.}
Cascades are the simplest underactuated coupled systems. Two-way couplings, and more generally non-triangular couplings between parabolic equations \cite{AKBGBdT_MCRF2011}, fall outside the present framework: the generator is no longer a nilpotent perturbation of a diagonal one, the double eigenvalues split, and the adapted block bases of Lemma \ref{rieszN} are not available. Two questions seem accessible. How stable is the description of $\mathcal{C}_T$ under a bounded or compact perturbation of the coupling? The difficulty is that such a perturbation acts on the weights themselves, and the example above shows that it may change $\mathcal{C}_T$ completely. How do the spaces $\mathcal{C}_T$ relate to stabilizability? Heat cascades are stabilizable, with an arbitrary decay rate, by finite-dimensional boundary feedbacks \cite{LPT_Automatica2026,LPT_SCL2026}, in state spaces from which they need not be null controllable at small time; for systems which are not time-reversible, the classical equivalence between exact controllability and stabilizability fails, and it would be useful to know which weight corresponds to which decay rate.

\paragraph{Other boundary conditions and other parabolic operators.}
System \eqref{system} was written with Dirichlet conditions, but the proofs use only the following properties of the scalar operator $\partial_{xx}$: an orthonormal basis of real eigenfunctions $\varphi_k$, simple eigenvalues $-\lambda_k$ with $\lambda_{k+1}-\lambda_k\asymp k$, and the lower bound $\inf_k\Vert\varphi_k\Vert_{L^2(\omega)}>0$. Neumann, Robin or mixed boundary conditions, and more generally Sturm--Liouville operators $\partial_x(a\partial_x)-q$ with regular coefficients, share these properties; the moment estimates of Proposition \ref{twosidedmoment} extend to such spectra \cite{BoyerMorancey_CRM2023}, and the corresponding block Gram matrices are computed in \cite{BoyerMorancey_AMBP2025}. Theorems \ref{main} and \ref{mainN} then hold with $e^{-2T\lambda_k}$ in place of $e^{-2Tk^2}$, with $\rho_k$ and $\Gamma_k$ computed in the new eigenbasis and the critical times defined by $\limsup(-\log\rho_k)/\lambda_k$. For Neumann conditions, the constant mode $\lambda_0=0$ is simply not dissipated. Only the explicit identification of Proposition \ref{geometricrho} depends on the sine basis: with Neumann conditions, $\varphi_k=\sqrt{2/\pi}\cos(kx)$ and the same Wronskian computation applies, whereas for a Sturm--Liouville operator the coefficient $A_k$ becomes an integral of $\beta\varphi_k^2$ against a variation-of-constants kernel. 

The same remarks apply to higher-order or nonlocal operators with real simple eigenvalues and growing gaps, such as $-\partial_{xxxx}$, the linearized Kuramoto--Sivashinsky operator $-\partial_{xxxx}-\nu\partial_{xx}$ (with hinged boundary conditions and $\nu$ such that the eigenvalues are simple; finitely many of them may be positive, which the moment method accommodates), or the fractional Laplacians $-(-\partial_{xx})^s$ with $s>\frac{1}{2}$ \cite{MicuZuazua_SICON2006}. The Gaussian scale becomes $e^{-2T\lambda_k}$ and the loss of the biorthogonal estimate becomes $e^{C\lambda_k^\theta}$, where $\theta<1$ is the exponent of the counting function; the analogue of Theorem \ref{main} holds as long as this loss is negligible with respect to $e^{-2T\lambda_k}$. 

The M\"untz condition $\sum_k\lambda_k^{-1}<+\infty$ is the limit of this approach; in the critical case $s=\frac{1}{2}$ of the half-heat equation, the space of null-controllable data is of a different nature \cite{HartmannKoenig_arXiv2025}. What is essential throughout is that the same scalar operator acts on every component, so that each eigenvalue carries an $N$-dimensional generalized eigenspace: with distinct diffusion coefficients, the eigenvalues $d_1k^2$ and $d_2k^2$ are no longer double, and null controllability then depends on the Diophantine properties of $d_1/d_2$ \cite{AKBGBdT_JFA2014,LucaDeTeresa_SeMA2013}.

\paragraph{Boundary control.}
For a Dirichlet control acting on the second equation at $x=0$, the natural state space is $H^{-1}(0,\pi;\R^2)$ and the adjoint observation is $\partial_xz_2(t,0)$. The modal calculation suggests the coefficient form $k^2(I_k^2\vert a_1\vert^2+\vert a_2\vert^2)$ and the known $H^{-1}$ null-controllability threshold
$$
T_\partial^\star=\limsup_{k\to +\infty}\frac{-\log\vert I_k\vert}{k^2}
$$
from \cite{AKBGBdT_JMAA2016}. A fixed-horizon analogue of Theorem \ref{main} should therefore involve the dual weights $k^{-2}(\vert c_1\vert^2/I_k^2+\vert c_2\vert^2)$ together with the $H^{-1}$ heat scale, with infinite cost in the first direction when $I_k=0$. We expect the same strategy to apply, but the extension is not formal: the boundary control operator is unbounded. The operator-range framework of Subsection \ref{sec_range_duality} extends to this setting, with $\mathcal{H}$ replaced by the extrapolation space $D(A^*)'$ and $L_T$ viewed as a bounded map into that space, whether or not the control operator is admissible \cite[Proposition 12.1.2]{TucsnakWeiss_book2009}; the hybrid range inclusions of \cite{Lissy_JMAA2026} may then help to compare the resulting spaces with the standard scales. We do not pursue these points here.

\paragraph{Higher space dimensions.}
For coupled heat systems, multidimensional observability has been obtained by Carleman estimates \cite{AKBD_JEE2009} and by the Lebeau--Robbiano method \cite{LissyZuazua_SICON2019}, under the respective structural hypotheses of those works. Hierarchical and transmutation methods also yield multidimensional controllability results for cascades with a reduced number of controls, including some configurations with disjoint coupling and control regions \cite{AlabauBoussouira_ADE2013}. These works establish observability or null controllability for a prescribed state space; they do not identify the full fixed-horizon space of null-controllable data.

Tensorized geometries are the most accessible next step. In rectangles and cylinders, some multidimensional constructions combine a one-dimensional moment problem with a transverse spectral inequality \cite{AKBGBMdT_arXiv2024,BBGO_SICON2014}; another tensorized approach reduces the system to a coupled family of block moment problems \cite{BoyerOlive_AIF2024}. Recent remarks on the higher-dimensional moment method are given in \cite{Morancey_HAL2026}. A fixed-horizon theorem would require two-sided estimates retaining the coupling-dependent cost of every transverse block. On a general bounded domain, multiplicities, small gaps, eigenfunction concentration and the conditioning of the coupling on each Laplace eigenspace suggest working with spectral clusters and operator-valued coefficient forms analogous to $\Gamma_k$.

Boundary-controlled scalar heat reachable states satisfy geometry-dependent holomorphic extension properties \cite{StrohmaierWaters_MZ2022}, while the essential maximal distance to the observation set gives an optimal exponential barrier for weighted heat observability in a broad metric-measure framework \cite{BoulardHayatTrelat_arXiv2026}. The fixed-horizon problem is to combine these geometric effects with the loss of observation inside spectral clusters. Under which spectral-cluster hypotheses does the quantitative Fattorini--Hautus principle of \cite{AKBGBM_JMPA2019} recover $T_{\mathcal{H}}$ from the quadratic behavior of the smallest cluster-observation singular value? Can different cluster directions have distinct thresholds? Which geometric and coupling-dependent quantities determine them?

\paragraph{Beyond linear parabolic cascades.}
Two directions leave the framework of this paper. 

The first is the dispersive analogue of the wave--heat cascades of \cite{LPT_arXiv2026}: a heat equation coupled with a linear Korteweg--de Vries equation. With periodic boundary conditions, the dispersive eigenvalues are $i\mu_m$ with $\mu_{m+1}-\mu_m\to+\infty$, so that Ingham's inequality holds in arbitrarily small time and no threshold is expected from the dispersive component; a mixed Ingham--M\"untz inequality and an adapted Riesz basis should then describe the null-controllable data, but for the boundary conditions of \cite{Rosier_COCV1997} the spectral analysis of the Korteweg--de Vries operator itself, including the multiplicities at the critical lengths, is a first obstacle. 

The second direction is nonlinear couplings. When one of the equations is nonlinear, a semilinear heat equation or the Korteweg--de Vries equation itself, there is no Riesz decomposition and the present method does not apply. Even for a semilinear perturbation of \eqref{system}, the question is whether the anisotropic space $\mathcal{C}_T$, rather than a smaller Sobolev-type space, is preserved by a fixed point or source-term argument \cite{LiuTakahashiTucsnak_COCV2013}, since a nonlinearity mixes the well-observed and the weakly observed directions of each spectral block. Both questions are open.

\appendix

\section{Appendix}\label{appendix}
\subsection{A vector-valued moment estimate of finite multiplicity}\label{appendixmoments}
The following estimate is used throughout the paper: with $N=2$ for the double blocks of Section \ref{twocascade}, and with arbitrary $N$ for the finite cascades of Section \ref{finitecascades}.

\begin{proposition}\label{twosidedmoment}
For every integer $N \geq 1$, there exists $\kappa_N>0$ such that, for every real Hilbert space $X$ and every $T>0$, there are positive constants $c_{T,N}$ and $C_{T,N}$ satisfying
\begin{equation}\label{twosidedmomentineq}
c_{T,N}\sum_{k\geq1}e^{-\kappa_Nk}\sum_{m=0}^{N-1}\Vert A_{m,k}\Vert_X^2
\leq \int_0^T\Big\Vert\sum_{k\geq1}e^{-k^2s}\sum_{m=0}^{N-1}s^mA_{m,k}\Big\Vert_X^2\,ds
\leq C_{T,N}\sum_{k\geq1}\frac{1}{k}\sum_{m=0}^{N-1}\Vert A_{m,k}\Vert_X^2
\end{equation}
for every finitely supported family $(A_{m,k})_{k\geq1,\,0\leq m<N}$ in $X$. The constants are independent of $X$ and of the family. Moreover, there exists $C_N>0$, independent of $T$ and $X$, such that
\begin{equation}\label{singlemoment}
\int_0^T e^{-2k^2s}\Big\Vert\sum_{m=0}^{N-1}s^mA_m\Big\Vert_X^2\,ds \leq \frac{C_N}{k^2}\sum_{m=0}^{N-1}\Vert A_m\Vert_X^2
\end{equation}
for every $k\geq1$ and every $(A_0,\ldots,A_{N-1})\in X^N$.
\end{proposition}

\begin{proof}
Fix $N\geq1$. For the lower bound, the scalar block moment theorem \cite[Theorem~46]{BoyerMorancey_CRM2023}, applied to the spectrum $\{k^2\mid k\geq1\}$ with singleton spectral groups and multiplicity $N$, gives a family $(q_{m,k})_{k\geq1,\,0\leq m<N}$ in $L^2(0,T)$ satisfying
$$
\int_0^T s^j e^{-\ell^2s}q_{m,k}(s)\,ds = \delta_{k\ell}\delta_{mj}\qquad \forall k,\ell\geq1\qquad \forall j,m\in\{0,\ldots,N-1\},
$$
and, for constants $K_N>0$ independent of $T$ and $M_{T,N}>0$,
\begin{equation}\label{biorthbound}
\Vert q_{m,k}\Vert_{L^2(0,T)} \leq M_{T,N}e^{K_Nk}\qquad \forall k\geq1\qquad \forall m\in\{0,\ldots,N-1\}.
\end{equation}
Indeed, the counting function of this spectrum has exponent $\frac{1}{2}$, so the exponential bound in that theorem has order $e^{K_Nk}$; the normalization by $(-s)^m/m!$ used there is absorbed into constants depending only on $N$. See also \cite[Theorem~V.4.26]{Boyer_Notes2023}. The underlying constructions originate in \cite{FattoriniRussell_ARMA1971,FattoriniRussell_QAM1975}; finite multiplicities are treated through matrix exponentials in \cite{AKBGBdT_JMPA2011}, and the multiplicity-two case is given in \cite[Lemma~3.1(b)]{FCGBdT_JFA2010}. These references provide complex-valued biorthogonal families; since the functions $s\mapsto s^je^{-\ell^2s}$ are real-valued, replacing each $q_{m,k}$ by its real part preserves the biorthogonality relations and does not increase its norm, so that the family may be taken real-valued.

Set $F(s)=\sum_ke^{-k^2s}\sum_{m=0}^{N-1}s^mA_{m,k}$. Bochner integration and biorthogonality give $A_{m,k}=\int_0^T F(s)q_{m,k}(s)\,ds$. By \eqref{biorthbound} and the Cauchy--Schwarz inequality,
$$
\Vert A_{m,k}\Vert_X^2 \leq M_{T,N}^2e^{2K_Nk}\Vert F\Vert_{L^2(0,T;X)}^2.
$$
Setting $\kappa_N = 2K_N+2$, multiplication by $e^{-\kappa_Nk}$ and summation in $m$ and $k$ prove the lower bound in \eqref{twosidedmomentineq}, with
$$
c_{T,N}^{-1} = NM_{T,N}^2\sum_{k\geq1}e^{-2k}.
$$

For the upper bound, let $(x_k)$ be a finite family in $X$. Expanding the square, integrating and using the Cauchy--Schwarz inequality gives
$$
\int_0^T\Big\Vert\sum_ke^{-k^2s}x_k\Big\Vert_X^2\,ds \leq \sum_{k,\ell}\frac{\Vert x_k\Vert_X\Vert x_\ell\Vert_X}{k^2+\ell^2}
=\sum_{k,\ell}m_{k,\ell}\frac{\Vert x_k\Vert_X}{\sqrt{k}}\frac{\Vert x_\ell\Vert_X}{\sqrt{\ell}},
$$
where $m_{k,\ell}=\frac{\sqrt{k\ell}}{k^2+\ell^2}$. For every $k\geq1$,
$$
\sum_{\ell\leq k}\frac{\sqrt{k\ell}}{k^2+\ell^2} \leq \frac{\sqrt{k}}{k^2}\sum_{\ell\leq k}\sqrt{\ell} \leq 1,\qquad
\sum_{\ell>k}\frac{\sqrt{k\ell}}{k^2+\ell^2} \leq \sqrt{k}\sum_{\ell>k}\ell^{-3/2} \leq \sqrt{k}\int_k^{+\infty}x^{-3/2}\,dx = 2.
$$
Thus $\sup_k\sum_{\ell\geq1}m_{k,\ell}\leq3$. The matrix $M=(m_{k,\ell})_{k,\ell\geq1}$ is symmetric, so its column sums satisfy the same bound. Schur's test \cite[Theorem~5.2]{HalmosSunder_book1978}, applied to $\N\setminus\{0\}$ with counting measure, yields $\Vert M\Vert_{L(\ell^2)}\leq3$. Hence
\begin{equation}\label{elementaryupper}
\int_0^T\Big\Vert\sum_ke^{-k^2s}x_k\Big\Vert_X^2\,ds \leq 3\sum_k\frac{\Vert x_k\Vert_X^2}{k}.
\end{equation}
For $F_m(s)=\sum_ke^{-k^2s}A_{m,k}$, we have $F(s)=\sum_{m=0}^{N-1}s^mF_m(s)$. Applying \eqref{elementaryupper} to each $F_m$ gives
$$
\Vert F\Vert_{L^2(0,T;X)}^2 \leq N\sum_{m=0}^{N-1}T^{2m}\Vert F_m\Vert_{L^2(0,T;X)}^2 \leq 3N\max(1,T^{2N-2})\sum_k\frac{1}{k}\sum_{m=0}^{N-1}\Vert A_{m,k}\Vert_X^2.
$$
This proves the upper bound in \eqref{twosidedmomentineq}, with $C_{T,N}=3N\max(1,T^{2N-2})$.

Finally, for a single frequency $k$, the inequality $\Vert\sum_{m=0}^{N-1}s^mA_m\Vert_X^2\leq N\sum_{m=0}^{N-1}s^{2m}\Vert A_m\Vert_X^2$ and the identity
$$
\int_0^{+\infty}s^{2m}e^{-2k^2s}\,ds = \frac{(2m)!}{(2k^2)^{2m+1}}
$$
give
$$
\int_0^T e^{-2k^2s}\Big\Vert\sum_{m=0}^{N-1}s^mA_m\Big\Vert_X^2\,ds \leq N\sum_{m=0}^{N-1}\frac{(2m)!}{(2k^2)^{2m+1}}\Vert A_m\Vert_X^2 \leq \frac{C_N}{k^2}\sum_{m=0}^{N-1}\Vert A_m\Vert_X^2,
$$
where one may take $C_N=N\max_{0\leq m<N}\frac{(2m)!}{2^{2m+1}}$. This proves \eqref{singlemoment}.
\end{proof}

The lower estimate extracts all $N$ coefficients at every frequency and controls cancellations between distinct frequencies and polynomial degrees. The upper estimate is the Hilbert--Schur bound; equivalently, \eqref{elementaryupper} says that $(\sqrt{k}e^{-k^2s})_{k\geq1}$ is a Bessel sequence in $L^2(0,T)$. All estimates are for fixed $N$; no uniformity as $N\to+\infty$ is asserted. For each fixed $N$, the cited biorthogonal construction gives $M_{T,N}\leq C_N'e^{C_N'/T}$ as $T\to0^+$, consistent with the classical fast-control scale \cite{BBGO_SICON2014,DardeErvedoza_APDE2019,TenenbaumTucsnak_JDE2007}. Estimates under weaker gap assumptions appear in \cite{CMV_DCDSS2020,GBOL_EECT2024}, and a matrix formulation is given in \cite[Proposition~9 and Theorem~10]{BoyerMorancey_CRM2023}. An elementary and self-contained proof of both bounds of \eqref{twosidedmomentineq}, through Blaschke products on the half-line for the biorthogonal family, Hermite interpolation for the blocks, and absorption of the tail for a finite horizon, is given in \cite[Chapter~4]{Trelat_Chengdu2026}.

\subsection{A triangular criterion for Riesz bases}\label{app_triangular}
The Riesz basis property of the adapted bases follows from an elementary perturbation argument, which we state once for a flag of $N$ blocks: it is used with $N=2$ in the proof of Lemma \ref{lem_Jordan}, and twice in the proof of Lemma \ref{rieszN} below.

\begin{lemma}[Triangular perturbation of an orthonormal basis]\label{lem_triangular}
Let $\mathcal{X}=\mathcal{X}_1\oplus\cdots\oplus\mathcal{X}_N$ be an orthogonal decomposition of a separable Hilbert space, and let $(e_{j,k})_{1\leq j\leq N,\,k\geq1}$ be an orthonormal basis of $\mathcal{X}$ with $e_{j,k}\in\mathcal{X}_j$ for all $j$ and $k$. Let $(\eta_{j,k})_{j,k}$ be a family of $\mathcal{X}$ such that
$$
\eta_{j,k}\in\mathcal{X}_1\oplus\cdots\oplus\mathcal{X}_{j-1}\quad\forall j,k,
\qquad\qquad
\sum_{j,k}\Vert\eta_{j,k}\Vert_{\mathcal{X}}^2<+\infty .
$$
Then $(e_{j,k}+\eta_{j,k})_{j,k}$ is a Riesz basis of $\mathcal{X}$, and its dual Riesz basis $(f_{j,k})_{j,k}$ satisfies
$$
f_{j,k}-e_{j,k}\in\mathcal{X}_{j+1}\oplus\cdots\oplus\mathcal{X}_N\qquad\forall j,k .
$$
The same statement holds, with the two flags exchanged, when $\eta_{j,k}\in\mathcal{X}_{j+1}\oplus\cdots\oplus\mathcal{X}_N$ for all $j$ and $k$.
\end{lemma}

\begin{proof}
Let $K\in L(\mathcal{X})$ be defined by $Ke_{j,k}=\eta_{j,k}$; it is Hilbert--Schmidt, and it maps $\mathcal{X}_j$ into $\mathcal{X}_1\oplus\cdots\oplus\mathcal{X}_{j-1}$. Hence $K^N=0$, the operator $\Id+K$ is invertible with $(\Id+K)^{-1}=\sum_{r=0}^{N-1}(-K)^r$, and $e_{j,k}+\eta_{j,k}=(\Id+K)e_{j,k}$ is the image of an orthonormal basis under an isomorphism, that is, a Riesz basis, whose dual basis is $f_{j,k}=(\Id+K)^{-*}e_{j,k}=\sum_{r=0}^{N-1}(-K^*)^re_{j,k}$. Finally, for $x\in\mathcal{X}_i$ and $y\in\mathcal{X}_j$, the product $\langle K^*x,y\rangle_{\mathcal{X}}=\langle x,Ky\rangle_{\mathcal{X}}$ vanishes unless $i<j$, so that $K^*$ maps $\mathcal{X}_i$ into $\mathcal{X}_{i+1}\oplus\cdots\oplus\mathcal{X}_N$ and $f_{j,k}-e_{j,k}=\sum_{r=1}^{N-1}(-K^*)^re_{j,k}$ has the announced form.
\end{proof}

The nilpotency of $K$ is what makes this elementary. For a perturbation which is only Hilbert--Schmidt, the Riesz property still holds when the perturbed family is complete and minimal, by the Bari criterion \cite{GohbergKrein_book1969}, which we shall not need.

\subsection{Proof of Lemma \ref{rieszN}}\label{app_proof_lem_rieszN}
In this subsection we write $A=A_0+Q$, where $A_0=\partial_{xx}\,\Id_{\R^N}$ and $Q=\sum_{1\leq i<j\leq N}\beta_{ij}E_{i,j}$, so that $(Qy)_i=\sum_{j>i}\beta_{ij}y_j$.

\paragraph{Construction of the primal vectors.}
Fix $k\geq1$. We construct $\Phi_{1,k},\ldots,\Phi_{N,k}$ and the entries of $J_k$ by induction on $j$, in such a way that
\begin{equation}\label{constructionN}
\Phi_{j,k}=\varphi_ke_j+\zeta_{j,k},\qquad (A+k^2)\Phi_{j,k}=\sum_{i<j}(J_k)_{ij}\Phi_{i,k},
\end{equation}
where the components of $\zeta_{j,k}$ of index $\geq j$ vanish, while those of index $<j$ belong to $H^2(0,\pi)\cap H_0^1(0,\pi)$ and are orthogonal to $\varphi_k$ in $L^2(0,\pi)$.

For $j=1$ we take $\Phi_{1,k}=\varphi_ke_1$, and \eqref{constructionN} holds since $Q\Phi_{1,k}=0$. Let $2\leq j\leq N$, and assume that $\Phi_{1,k},\ldots,\Phi_{j-1,k}$ have been constructed. Writing $\Phi_{j,k}=\varphi_ke_j+\sum_{i<j}u_ie_i$, the components of index $\geq j$ in \eqref{constructionN} hold for every choice of the functions $u_i$, because $(\partial_{xx}+k^2)\varphi_k=0$ and because $\Phi_{i,k}$ has no component of index $\geq j$ for $i<j$. The component of index $i<j$ reads
$$
(\partial_{xx}+k^2)u_i=(J_k)_{ij}\varphi_k+g_i,
\qquad
g_i=\sum_{i<m<j}(J_k)_{mj}(\Phi_{m,k})_i-\beta_{ij}\varphi_k-\sum_{i<l<j}\beta_{il}u_l .
$$
We solve these equations successively for $i=j-1,j-2,\ldots,1$; at each step, $g_i$ involves only quantities determined at the previous steps. As for \eqref{psik}, the Dirichlet problem $(\partial_{xx}+k^2)u=f$, $u(0)=u(\pi)=0$, has a solution if and only if $\langle f,\varphi_k\rangle_{L^2(0,\pi)}=0$. This forces $(J_k)_{ij}=-\langle g_i,\varphi_k\rangle_{L^2(0,\pi)}$, and then $u_i$ is unique in the orthogonal of $\varphi_k$, with
$$
\Vert u_i\Vert_{L^2(0,\pi)}\leq\frac{1}{2k-1}\big\Vert(J_k)_{ij}\varphi_k+g_i\big\Vert_{L^2(0,\pi)} .
$$
A descending induction on $i$ gives $\vert(J_k)_{ij}\vert\leq C$ and $\Vert u_i\Vert_{L^2(0,\pi)}\leq C/k$, where $C$ depends only on $N$ and on the norms $\Vert\beta_{ij}\Vert_{L^\infty(0,\pi)}$. This proves the first identity in \eqref{similarityN}, with a strictly upper triangular matrix $J_k$ satisfying $\sup_k\Vert J_k\Vert<+\infty$ and $J_k^N=0$, together with the primal half of \eqref{flagnormalizationN} and the estimate $\Vert\zeta_{j,k}\Vert_{\mathcal{H}}=\mathrm{O}(k^{-1})$. Adding to $\Phi_{j,k}$ a linear combination of the $\Phi_{i,k}$, $i<j$, with coefficients bounded uniformly in $k$, preserves all these properties; this is how the coefficients $\delta_k$ enter Lemma \ref{lem_Jordan}.

The adjoint operator is $A^*=A_0+Q^*$, with $(Q^*y)_j=\sum_{i<j}\beta_{ij}y_i$: this is again a cascade, for the reversed order of the components. The same construction, run on the components in reverse order, provides vectors $\widetilde\Psi_{j,k}=\varphi_ke_j+\widetilde\xi_{j,k}$, where $\widetilde\xi_{j,k}$ has nonzero components of index $>j$ only and $\Vert\widetilde\xi_{j,k}\Vert_{\mathcal{H}}=\mathrm{O}(k^{-1})$, and a strictly lower triangular matrix $\widetilde J_k$, uniformly bounded in $k$, such that $A^*\widetilde\Psi_k=\widetilde\Psi_k(-k^2\Id_{\R^N}+\widetilde J_k)$.

\paragraph{Riesz bases and biorthogonality.}
We apply Lemma \ref{lem_triangular} with $\mathcal{X}_j=L^2(0,\pi)e_j$, with the orthonormal basis $e_{j,k}=\varphi_ke_j$ of $\mathcal{H}$ and with the corrections $\zeta_{j,k}$, which are square summable by the above estimate: the family $(\Phi_{j,k})_{j,k}$ is a Riesz basis of $\mathcal{H}$. The reversed form of the same lemma shows that $(\widetilde\Psi_{j,k})_{j,k}$ is a Riesz basis as well.

Let $k,\ell\geq1$ and let $M$ be the matrix with entries $M_{ij}=\langle\Phi_{j,k},\widetilde\Psi_{i,\ell}\rangle_{\mathcal{H}}$. All these vectors belong to $D(A)=D(A^*)$, so that the identities $\langle A\Phi_{j,k},\widetilde\Psi_{i,\ell}\rangle_{\mathcal{H}}=\langle\Phi_{j,k},A^*\widetilde\Psi_{i,\ell}\rangle_{\mathcal{H}}$ read
$$
M(-k^2\Id_{\R^N}+J_k)=(-\ell^2\Id_{\R^N}+\widetilde J_\ell)^*M .
$$
When $\ell\neq k$, the two matrices involved have the single eigenvalues $-k^2$ and $-\ell^2$, which are distinct, and this Sylvester equation forces $M=0$. When $\ell=k$, the flag relations give $M_{ij}=0$ for $i>j$ and $M_{jj}=1$, so that $M$ is invertible and $M-\Id_{\R^N}=\mathrm{O}(k^{-1})$. Setting $\Psi_k=\widetilde\Psi_kM^{-*}$, that is,
$$
\Psi_{i,k}=\widetilde\Psi_{i,k}+\sum_{m>i}(M^{-*})_{mi}\widetilde\Psi_{m,k},
$$
we obtain a family which is biorthogonal to $(\Phi_{j,k})_{j,k}$, which still satisfies the dual half of \eqref{flagnormalizationN} because $M^{-*}$ is lower triangular with unit diagonal, and which is still a Riesz basis because $M^{-*}$ and its inverse are bounded uniformly in $k$. It is therefore the dual Riesz basis of $(\Phi_{j,k})_{j,k}$. Finally, the span of $\Psi_{1,k},\ldots,\Psi_{N,k}$ is the span of $\widetilde\Psi_{1,k},\ldots,\widetilde\Psi_{N,k}$, which is invariant under $A^*$, so that $A^*\Psi_k=\Psi_kS_k$ for some matrix $S_k$; pairing with $\Phi_{j,k}$ and using biorthogonality gives $(S_k)_{ji}=(-k^2\Id_{\R^N}+J_k)_{ij}$, which is the second identity in \eqref{similarityN}.

\paragraph{The generalized eigenspaces.}
By \eqref{constructionN} and the strict triangularity of $J_k$, we have $(A+k^2)^N\Phi_{j,k}=0$, so that $\Span(\Phi_{1,k},\ldots,\Phi_{N,k})\subset E_k$. Conversely, let $y\in E_k$, say $(A+k^2)^my=0$ for some $m\geq1$. By \eqref{similarityN}, every $\Psi_{i,\ell}$ belongs to the domain of every power of $A^*$, and
$$
(A^*+k^2)^m\Psi_\ell=\Psi_\ell\big((k^2-\ell^2)\Id_{\R^N}+J_\ell^*\big)^m .
$$
For $\ell\neq k$, the matrix $(k^2-\ell^2)\Id_{\R^N}+J_\ell^*$ is invertible, hence so is its $m$-th power, and the vectors $(A^*+k^2)^m\Psi_{i,\ell}$, $1\leq i\leq N$, span the same space as the vectors $\Psi_{i,\ell}$. Since
$$
\langle y,(A^*+k^2)^m\Psi_{i,\ell}\rangle_{\mathcal{H}}=\langle(A+k^2)^my,\Psi_{i,\ell}\rangle_{\mathcal{H}}=0 ,
$$
all the coordinates of $y$ outside the $k$-th block vanish, that is, $y\in\Span(\Phi_{1,k},\ldots,\Phi_{N,k})$. Hence $\dim E_k=N$, and the same argument applied to $A^*$ gives $\dim E_k^*=N$.

\end{document}